\documentclass[10pt,twoside]{amsart}
\usepackage{latexsym,amssymb,graphicx}
\usepackage{mathrsfs}%,txfonts}%dsfont,
\usepackage[all]{xy}

\let\oldtocsection=\tocsection

\let\oldtocsubsection=\tocsubsection

\let\oldtocsubsubsection=\tocsubsubsection

\renewcommand{\tocsection}[2]{\hspace{0em}\oldtocsection{#1}{#2}}
\renewcommand{\tocsubsection}[2]{\hspace{1em}\oldtocsubsection{#1}{#2}}
\renewcommand{\tocsubsubsection}[2]{\hspace{2em}\oldtocsubsubsection{#1}{#2}}

\newtheorem{theorem}{Theorem}[section]
\newtheorem{lemma}[theorem]{Lemma}
\newtheorem{corollary}[theorem]{Corollary}
\newtheorem{proposition}[theorem]{Proposition}

\theoremstyle{definition}

\newtheorem{definition}[theorem]{Definition}
\newtheorem{example}[theorem]{Example}

\newtheorem{remark}[theorem]{Remark}

\numberwithin{equation}{section}

\newcommand{\Z}{\mathbb{Z}}

\newcommand{\B}{\mathcal{B}}

\newcommand{\D}{\mathcal{D}}

\newcommand{\N}{\mathbb{N}}

\newcommand{\ffi}{\varphi}
\newcommand{\eps}{\varepsilon}

\newcommand{\im}{\operatorname{Im}}

\newcommand{\coker}{\operatorname{coker}}

\newcommand{\SL}{\operatorname{SL}}

\newcommand{\Tor}{\operatorname{Tor}}

\newcommand{\Sp}{\operatorname{Sp}}

\newcommand{\GL}{\operatorname{GL}}

\newcommand{\Skew}{\operatorname{Skew}}
\newcommand{\Pf}{\operatorname{Pf}}

\newcommand{\ann}{\operatorname{Ann}}

\renewcommand{\omit}{\text{\tiny ${\wedge}$}}

\title[Homology of symplectic groups]{On the Homology stability range for symplectic groups}
 \author{Marco Schlichting}
 \address{Marco Schlichting, Mathematics Institute,
Zeeman Building,
University of Warwick,
Coventry CV4 7AL, UK} 

\thanks{}

\email{m.schlichting@warwick.ac.uk}



\begin{document}
\bibliographystyle{alpha}

\begin{abstract}
We improve, by a factor of $2$, known homology stability ranges for the integral homology of symplectic groups over commutative local rings with infinite residue field and show that the obstruction to further stability is bounded below by Milnor-Witt $K$-theory.
In particular our stability range is optimal in many cases.
 \end{abstract}

\maketitle

\tableofcontents

\section{Introduction}

This paper addresses the question of optimal homology stability for symplectic groups over local rings.
Recall that the symplectic group $\Sp_{2n}(R)$ over a commutative ring $R$ is the group of $R$-linear automorphisms $A$ of $R^{2n}$ that preserve the standard symplectic inner product, that is, $\langle Ax,Ay\rangle = \langle x, y\rangle$ for all $x=(x_1,x_2,...,x_{2n}),\ y=(y_1,y_2,...,y_{2n})\in R^{2n}$ where $\langle x, y \rangle =  \sum_{i=1}^n (x_{2i+1}y_{2i+2} - x_{2i+2}y_{2i+1})$.
We consider $\Sp_{2n}(R)$ as a subgroup of $\Sp_{2n+2}(R)$ by means of the embedding $A \mapsto \left(\begin{smallmatrix} 1_{R^2} & 0 \\ 0 & A\end{smallmatrix}\right)$.
%Our goal is to study the effect of that inclusion in homology.
The following is part of Theorem \ref{thm:OptStabilityText} in the text.
All homology groups in this paper are with integer coefficients unless indicated otherwise.

\begin{theorem}
\label{thm:OptStabilityIntro}
Let $R$ be a commutative local ring with infinite residue field and $n\geq 1$ an integer.
Then the relative integral homology groups satisfy
\begin{equation}
\label{eqn:StabMap}
H_d(\Sp_{2n}(R), \Sp_{2n-2}(R)) =0,\hspace{5ex}d<2n.
\end{equation}
\end{theorem}

In particular, for all integers $n \geq 0$ inclusion of groups induces isomorphisms
\begin{equation}
\label{eqn:H2nStability}
H_{2n}(\Sp_{2n}R) \stackrel{\cong}{\longrightarrow}  H_{2n}(\Sp_{2n+2}R) \stackrel{\cong}{\longrightarrow} H_{2n}(\Sp_{2n+4}R) \stackrel{\cong}{\longrightarrow} \cdots
\end{equation}
and a surjection followed by isomorphisms
\begin{equation}
\label{eqn:H2n+1Stability}
H_{2n+1}(\Sp_{2n}R) \twoheadrightarrow  H_{2n+1}(\Sp_{2n+2}R) \stackrel{\cong}{\longrightarrow} H_{2n+1}(\Sp_{2n+4}R) \stackrel{\cong}{\longrightarrow} \cdots.
\end{equation}

For $n=1$, the isomorphisms (\ref{eqn:H2nStability}) where proved by van der Kallen \cite{vdKallen}
% for local rings with sufficiently many elements in the residue field \cite{vdK:H2Sp} and 
generalizing the results of Matsumoto \cite{Matsumoto} for infinite fields.
In joint work with Sarwar \cite{SarwarMe}, we proved (\ref{eqn:H2n+1Stability}) for $n=1$.
Mirzaii \cite{Mirzaii} proves that the relative homology groups in (\ref{eqn:StabMap}) vanish for $d<n-1$.
For infinite fields,  Essert  \cite{Essert} and Sprehn-Wahl \cite{SprehnWahl} prove the vanishing of that group for $d< n$. 
Thus, our result improves the best known stability ranges by a factor of two.
\vspace{2ex}

For a commutative local ring $R$ with infinite residue field, consider the graded $\Z[R^*]$-algebra generated in degree $1$ by the augmentation ideal $I[R^*]\subset \Z[R^*]$ modulo the Steinberg relation $[a]\otimes[1-a]$ for $a,1-a\in R^*$.
For $n\geq 2$, the $n$-th degree part of that algebra is the $n$-th Milnor-Witt $K$-group $K^{MW}_n(R)$ of $R$ \cite[\S 4]{myEuler} which was first defined in \cite{Morel:book} for fields where it
plays an important role in $\mathbb{A}^1$-homotopy theory.
The following is Theorem \ref{thm:ObstructionSurjection} in the text.

\begin{theorem}
\label{thm:OptStabilityIntro2}
Let $R$ be a commutative local ring with infinite residue field and $n\geq 1$ an integer.
Then the inclusion $\Sp_{2n}(R) \subset \SL_{2n}(R)$ induces a surjection
$$H_{2n}(\Sp_{2n}R, \Sp_{2n-2}R) \twoheadrightarrow H_{2n}(\SL_{2n}R, \SL_{2n-1}R) \cong K^{MW}_{2n}(R).$$
\end{theorem}

In particular, the homology stability range in Theorem \ref{thm:OptStabilityIntro} is optimal as soon as the Milnor-Witt $K$-theory group $K^{MW}_{2n}(R)$ is non-trivial.
This happens, for instance, when the residue field of $R$ has a real embedding.
For many infinite fields, the surjection $H_{4}(\Sp_{4}R, \Sp_{2}R) \twoheadrightarrow  K^{MW}_{4}(R)$ is not injective; see Remark \ref{rmk:NotInjective}.
\vspace{1ex}

The strategy for proving our homology stability  range is classical.
We construct a highly connected chain complex on which our groups act and study the resulting spectral sequences.
The chain complex we use is essentially that of \cite{SarwarMe}.
In {\it loc.\hspace{-1ex} cit.}\hspace{-.5ex} we were not able to prove degeneration of the spectral sequence.
This is what is achieved here.
Our innovation is the Limit Theorem \ref{thm:LimitThm} which gives a criterion for the vanishing of certain modules built out of relative homology groups that carry an action of the multiplicative monoid $(R,\cdot,1)$ of a ring $R$ and may be useful for groups other than $\Sp_{2n}(R)$; see the examples in Section \ref{sec:LimitThm}.

\section{Non-degenerate unimodular sequences}

In this section we review notation and a few results from \cite{SarwarMe}.
\vspace{1ex}

Throughout this paper, $n\geq 0$ will be an integer, $R$ will be a commutative local ring with infinite residue field, $R^*$ its group of units, $\GL_n(R)$ the group of invertible $n\times n$ matrices with entries in $R$, 
$$\psi_{2n} = \psi_2 \perp \cdots \perp \psi_2 = \left(\begin{smallmatrix}\psi_2 &&&\\ & \psi_2 && \\ && \ddots & \\ &&& \psi_2\end{smallmatrix}\right) = \bigoplus_1^n\psi_2,\hspace{3ex} \psi_2 = \left(\begin{smallmatrix}0 & 1 \\ -1 & 0 \end{smallmatrix}\right)$$
the standard hyperbolic symplectic form of rank $2n$, 
$$\Sp_{2n}(R) =\{ A \in \GL_{2n}(R)|\ {^t\!A}\, \psi_{2n}\, A = \psi_{2n}\}$$
the symplectic group or rank $2n$, considered as a subgroup of $\Sp_{2n+2}(R)$ by means of the embedding 
\begin{equation}
\label{eqn:GSpIncls}
\Sp_{2n}(R) \subset \Sp_{2n+2}(R): A \mapsto \left(\begin{smallmatrix}1 & 0 & 0 \\ 0 & 1 & 0 \\ 0 & 0 & A\end{smallmatrix}\right).
\end{equation}
For the purpose of this paper, the {\em symplectic group of rank $2n+1$} is the subgroup
$$\Sp_{2n+1}(R) = \{A \in \Sp_{2n+2}(R)|\ Ae_1 = e_1\}$$
of $\Sp_{2n+2}(R)$ fixing the first standard basis vector $e_1$.
This is the group of matrices
% under multiplication
\begin{equation}
\label{eqn:GSpodd}
\left(\begin{smallmatrix}1& c & {^t\!u}\psi M  \\ 0 & 1 & 0 \\  0 & u & M \end{smallmatrix}\right)
\end{equation}
where $\psi = \psi_{2n}$, $M\in \Sp_{2n}(R)$, $u\in R^{2n}$, $c\in R$.
The inclusions (\ref{eqn:GSpIncls}) refine to 
the sequence of inclusions of groups
\begin{equation}
\label{eqn:SpInclusions}
1=\Sp_0(R) \subset \Sp_1(R) \subset \Sp_2(R) \subset \dots \subset \Sp_n(R) \subset \Sp_{n+1}(R)\subset \dots
\end{equation}
where
$$
\Sp_{2n}(R) \subset \Sp_{2n+1}(R): M \mapsto \left(\begin{smallmatrix}1 & 0 & 0 \\ 0 & 1 & 0 \\ 0& 0 & M \end{smallmatrix}\right), \hspace{3ex} \Sp_{2n-1} (R)\subset \Sp_{2n}(R): M \mapsto M.
$$
In Theorem \ref{thm:OptStabilityText} below we study homology stability for the sequence of groups (\ref{eqn:SpInclusions}).
We shall denote the inclusions $\Sp_{r}(R) \subset \Sp_s(R)$ by $\eps^s_r$, or simply by $\eps$ if source and target group are understood, $r\leq s$.
Small rank symplectic groups are as follows
$$\Sp_0(R) = \{1\},\hspace{2ex} \Sp_1(R) =\left \{\left(\begin{smallmatrix}1 & c \\ 0 & 1\end{smallmatrix}\right)|\  c\in R\right\},\hspace{2ex}\Sp_2(R) = \SL_2(R).$$

Let $0\leq q$ be an integer. 
We denote by $\Skew_q(R)$ the set of $q\times q$ {\em skew symmetric} matrices with entries in $R$, that is those matrices $A=(a_{ij})$ such that $a_{ij}=-a_{ji}$, $a_{ii}=0$, $a_{ij}\in R$, $1 \leq i,j \leq q$.
We denote by 
$$\Skew_q^+(R) \subset \Skew_q(R)$$
 the subset of {\em non-degenerate skew-symmetric matrices}, that is those matrices $A\in \Skew_q(R)$ such that for all subsets $I\subset \{1,...,q\}$ of even cardinality the matrix $A_I$, obtained from $A$ deleting all rows and columns not in $I$, is invertible.

The $R$-module $R^{2n}$ will aways be equipped with the standard symplectic bilinear form 
$\langle x, y \rangle =  \sum_{i=1}^n (x_{2i+1}y_{2i+2} - x_{2i+2}y_{2i+1})$ where $x={^t}(x_1,x_2,...,x_{2n})$, $y={^t}(y_1,y_2,...,y_{2n}) \in R^{2n}$.
The {\em Gram matrix} $\Gamma(v)$ of a sequence $v=(v_1,...,v_q)$ of $q$ vectors $v_1,...,v_q \in R^{2n}$ is the skew symmetric $q\times q$ matrix 
$$\Gamma(v)=(\langle v_i,v_j\rangle )^q_{i,j=1} = {^tv}\ \psi_{2n}\ v$$
with $(i,j)$ entry $\langle v_i,v_j\rangle$.
A sequence $v=(v_1,...,v_q)$ of $q$ vectors in $R^{2n}$ is called
 {\em unimodular} if each subsequence of length $r \leq \min(q,2n)$ 
is a basis of a direct summand of $R^{2n}$. 
A unimodular sequence $v=(v_1,...,v_q)$ of vectors in $R^{2n}$ is called {\em non-degenerate} if for all subsets $I\subset \{1,...q\}$ of even cardinality $|I| \leq \min(q,2n)$, the Gram matrix $\Gamma(v_I)$ is invertible, where $v_I$ is the sequence of vectors obtained from $v$ by deleting all columns not in $I$.
We denote by 
$$U_q(R^{2n}) = \{v =(v_1,...,v_q)|\ v\ \text{non-degenerate unimodular in } R^{2n}\}$$
the set of non-degenerate unimodular sequences 
 of length $q$ in $R^{2n}$.
The set $U_0(R^{2n})$ is the singleton set consisting of the empty sequence, and the set $U_q(R^0)$ is the singleton set with unique element the sequence $(0,0,...,0)$ of length $q$.
The symplectic group $\Sp_{2n}(R)$ acts from the left on $U_q(R^{2n})$ by matrix multiplication $Av=(Av_1,...,Av_q)$ for $A\in \Sp_{2n}(R)$, $v=(v_1,...,v_q)\in U_{q}(R^{2n})$.
Note that the Gram matrix of $v$ and $Av$ are the same for all $A\in \Sp_{2n}(R)$.
The following was proved in \cite[\S 2]{SarwarMe}.

\begin{lemma}
\label{lem:GammaBij}
Let $R$ be a local ring.
Then for all integers $0 \leq q \leq 2n+1$ the Gram matrix defines a bijection
	$$\Gamma:\Sp_{2n}(R)\backslash U_q(R^{2n}) \stackrel{\cong}{\longrightarrow} \Skew_q^+(R).$$
\end{lemma}

\begin{definition}
Let $R$ be a local ring and $n,q\geq 1$ be integers.
A non-degenerate unimodular sequence $u=(u_1,...,u_q) \in U_q(R^{2n})$ is said to be {\em in normal form} 
if for $r=\min(2n,q)$, the matrix $(u_1,...,u_r)$ is upper triangular, $(u_i)_i=1$ for $i $ odd and $(u_i)_{i-1}=0$ for $i$ even, $i=1,...,r$.
\end{definition}

In this paper, we will identify $R^q$ with the subspace of $R^{2n}$ sending the standard basis vector $e_i$ of $R^q$ to the standard basis vector $e_i$ of $R^{2n}$, $i=1,...,q$.
Note that if $q\leq 2n$ and $u\in U_q(R^{2n})$ is in normal form then $u$ spans $R^q$.

\begin{lemma}
\label{lem:normalForm}
Let $R$ be a local ring and $n,q\geq 1$ be integers with $q\leq 2n+1$.
Then for every $A\in \Skew_q^+(R)$, there is a non-degenerate unimodular sequence $u\in U_q(R^{2n})$ which is in normal form and such that $\Gamma(u)=A$.
\end{lemma}

In the situation of Lemma \ref{lem:normalForm}, we will call $u$ a {\em normal form of $A$}.

\begin{proof}[Proof of Lemma \ref{lem:normalForm}]
This is proved by induction on $q\geq 1$.
The case $q=1$ is clear, choosing $u_1=e_1$.
Assume we are given $A\in \Skew_{q+1}(R)$ and $u\in U_{q}(R^{2n})$ generating $R^{q}$, for instance, $u$ is in normal form, such that $\Gamma(u)=A_{\{1,...,q\}}$ where for $I\subset \{1,...,q+1\}$ we write $A_I$  for the skew symmetric matrix obtained from $A$ by deleting all rows and columns not in $I$.
Then $u$ is a basis in $R^q$ and thus defines an invertible $q\times q$ matrix.
If $q$ is even, then there is a unique $x\in R^{q}$ such that $\Gamma(u,x)=A_{\{1,...,q+1\}}$, namely, the solution to $^tu\psi_qx=v$ where $v$ is the $q+1$st column of $A$ with last row removed.
If $q=2n$, set $u_{q+1}=x$.
If $q<2n$, set $u_{q+1}=x+e_{q+1}$.
If $q$ is odd, then $q-1$ is even and we let $x\in R^{q-1}$ be the unique solution to 
$\Gamma(u_1,...,u_{q-1},x)=A_{\{1,...q-1,\hat{q},q+1\}}$
and set $u_{q+1}=x + \alpha e_{q+1}$ where $\alpha$ is the $(q,q+1)$-entry of $A$.
\end{proof}

For a set $S$, we denote by $\Z[S]$ the free abelian group with basis $S$.
We make the graded abelian group 
\begin{equation}
\label{eqn:ZU*cx}
\Z[U_*(R^{2n})] = \{\Z[U_q(R^{2n})],q\geq 0\}
\end{equation}
into a chain complex with differential $d:\Z[U_q(R^{2n})] \to \Z[U_{q-1}(R^{2n})]$ defined on basis elements $(v_1,...,v_q)$ by 
$$dv = \sum_{i=1}^q(-1)^{i+1}d_iv$$
where $d_iv=v^{\wedge}_{{i}}=(v_1,...,\hat{v}_i,...,v_q)$ is obtained from $v$ by deleting the $i$-th vector $v_i$.
The following was proved in \cite[\S 2]{SarwarMe}.

\begin{lemma}
	\label{lem:UExact}
	Let $R$ be a local ring with infinite residue field and $n\geq 0$ an integer.
	Then the chain  complex $(\Z[U_*(R^{2n})],d_*)$ is acyclic.
	 That is, for all $p\in \Z$ we have
	$$H_p(\Z[U_*(R^{2n})])=0.$$
\end{lemma}

Similarly, 
we make the graded abelian group $\Z[\Skew^+_*(R)]$ into a chain complex with differential $d:\Z[\Skew^+_q(R)] \to \Z[\Skew^+_{q-1}(R)]$ defined on basis elements $A\in \Skew^+_q(R)$ by 
$$dA = \sum_{i=1}^q(-1)^{i+1}d_iA$$
where $d_iA=A^{\wedge}_{{i}}$ is obtained from $A$ by deleting the $i$-th row and column.
The following was again proved in \cite[\S 2]{SarwarMe}.

\begin{lemma}
	\label{lem:SkewExact}
	Let $R$ be a local ring with infinite residue field.
	Then the chain  complex $(\Z[\Skew_*^+(R)],d_*)$ is acyclic.
	 That is, for all $p\in \Z$ we have
	$$H_p(\Z[\Skew_*^+(R)])=0.$$
\end{lemma}

\section{The spectral sequence and its $E^1$-page}

In this section we introduce the spectral sequence (\ref{eqn:E1spseq}) which leads to our homological stability range in Theorem \ref{thm:OptStabilityIntro} and identify its $E^1$-term.
\vspace{1ex}

For a complex $M_*$ of abelian groups and an integer $r\in \Z$, we denote by $M_{\leq r} \subset M_*$ the subcomplex which is $(M_{\leq r})_i=M_i$ for $i\leq r$ and $(M_{\leq r})_i=0$ for $i>r$.
We call the resulting filtration $\cdots \subset M_{\leq r-1} \subset M_{\leq r} \subset M_{\leq r+1} \subset \cdots$ of $M_*$, the {\em filtration by degree}.
The filtration by degree
$$C_{\leq 0}(R^{2n}) \subset C_{\leq 1}(R^{2n}) \subset \dots \subset C_{\leq 2n-1}(R^{2n}) \subset C_{\leq 2n}(R^{2n})=C_*(R^{2n})$$
of the complex 
$$C_*(R^{2n}) = \Z[U_{\leq 2n}(R^{2n})]$$
of $\Sp_{2n}(R)$-modules
yields the exact sequence of complexes
$$0 \to C_{\leq q-1}(R^{2n}) \to C_{\leq q}(R^{2n}) \to C_{\leq q}(R^{2n})/C_{\leq q-1}(R^{2n}) \to 0.$$
Upon applying the functor $H_*(\Sp_{2n},\phantom{fd}) = \Tor_*^{\Sp_{2n}}(\Z,\phantom{fd})$, the exact sequences yield the exact couple 
$D^1_{p+1,q-1} \to D^1_{p,q} \to E^1_{p,q} \to D^1_{p,q-1}$
where
$$D^1_{p,q}=H_{p+q}(\Sp_{2n}(R),C_{\leq q}(R^{2n})),\hspace{3ex} E^1_{p,q}=H_{p+q}(\Sp_{2n}(R),C_{\leq q}(R^{2n})/C_{\leq q-1}(R^{2n}))$$
and hence the spectral sequence
\begin{equation}
\label{eqn:E1spseq}
E^1_{p,q}=H_p(\Sp_{2n}(R),C_q(R^{2n})) \Rightarrow H_{p+q}(\Sp_{2n}(R),C_*(R^{2n}))
\end{equation}
with differential $d^r_{p,q}$ of bidegree $(r-1,-r)$.

The following lemma shows that the abutment of the spectral sequence (\ref{eqn:E1spseq}) vanishes in degrees $p+q<2n$.

\begin{lemma}
\label{lem:AbuttVanishing}
Let $R$ be a local ring with infinite residue field. 
Then
$$H_{i}(\Sp_{2n}(R),C_*(R^{2n})) = 0,\hspace{3ex}i<2n.$$
\end{lemma}

\begin{proof}
Let $M$ be the kernel of $d:C_{2n}(R^{2n}) \to C_{2n-1}(R^{2n})$.
By Lemma \ref{lem:UExact}, the inclusion of complexes $M[2n] \to C_*(R^{2n})$ is a quasi-isomorphism.
In particular, 
$$H_{i}(\Sp_{2n}(R),C_*(R^{2n})) = H_{i}(\Sp_{2n}(R),M[2n])= H_{i-2n}(\Sp_{2n}(R),M).$$
The result follows since for all $G$-modules $M$, we have $H_j(G,M)=0$ for $j<0$.
\end{proof}

\begin{remark}
In \cite{SarwarMe}, we studied the spectral sequence associated with the complex $\Z[U_{\leq 2n+1}(R^{2n})]$ and its filtration by degree.
\end{remark}

Let $0 \leq q \leq 2n$ be integers.
%We identify $R^q$ with the subspace of $R^{2n}$ sending the standard basis vector $e_i$ of $R^q$ to the standard basis vector $e_i$ of $R^{2n}$, $i=1,...,q$.
Let $v\in U_q(R^{2n})$ be a non-degenerate unimodular sequence which spans $R^q$.
Note that for every $A\in \Skew_q(R)$ there is such a $v$ with $\Gamma(v)=A$, for instance a normal form of $A$ will do; see Lemma \ref{lem:normalForm}.
As an ordered basis of $R^q$, $v$ defines an element of $\GL_q(R)$ and as such has a determinant $\det(v)\in R^*$.
Using the standard functoriality of group homology as in \cite[III.8]{Br82}, we define a map 
$$f_v: H_p(\Sp_{2n-q}(R);\Z) \longrightarrow H_p(\Sp_{2n}(R);\Z[U_q(R^{2n})])$$
by
$$
f_v = \left\{ 
\renewcommand\arraystretch{1.5}
\begin{array}{lll} 
(\eps,v)_* & 0 \leq q \leq 2n,& q \text{\ even}\\
(\eps \circ c_{\det v}, v)_* & 0 \leq q \leq 2n,& q \text{ odd}
\end{array}\right.
$$
where $\eps:\Sp_{2n-q} (R) \to \Sp_{2n}(R)$ is the standard embedding, $v$ denotes the homomorphism of abelian groups $\Z \to \Z[U_q(R^{2n})]$ sending $1$ to $v$, and for $a\in R^*$, $c_a:\Sp_{2n-q}(R) \to \Sp_{2n-q}(R)$ is conjugation $A \mapsto D A D^{-1}$ with the diagonal  matrix
$$D= \left(\begin{smallmatrix}a & 0 & 0 \\ 0 & a^{-1} & 0 \\ 0 & 0 & 1_{2n-q-1}\end{smallmatrix}\right) \in \Sp_{2n-q+1}(R)$$
for $0<  q <2n$ odd.

\begin{lemma}
\label{lem:vIndep}
Let $q$ be an integer such that $0 \leq q \leq 2n$.
Let $u,v\in U_q(R^{2n})$ be non-degenerate unimodular sequences that span $R^q$.
If $\Gamma(u) = \Gamma(v)$ then $f_u=f_v$.
\end{lemma}

\begin{proof}
If $q$ is even, then the $R$-linear automorphism $B$ of $R^q$ sending $u$ to $v$ is an isometry, since $\Gamma(u)=\Gamma(v)$.
We extend $B$ to an isometry of $R^{2n}$ by requiring $B e_i=e_i$ for $i=q+1,...,2n$.
Since $B\in \Sp_{2n}(R)$ commutes with every element of $\Sp_{2n-q}(R)$, we have $f_u=f_v$.

Assume now $q=2r+1$ odd, $0 \leq r < n$.
We consider $u,v$ as elements in $\GL_q(R)$.
There are unique vectors $x,y\in R^{q+1}$ such that
$$\left(\begin{array}{c|c}{^t}u& 0 \\ \hline 0 & 1\end{array}\right)\psi_{q+1}\ x=\left(\begin{array}{c|c}{^t}v& 0 \\ \hline 0 & 1\end{array}\right)\psi_{q+1}\ y = e_q\in R^{q+1}$$
since the $(q+1)\times (q+1)$ matrices involved are invertible.
Then $\Gamma(u,x)=\Gamma(v,y)$, by definition of $x$ and $y$.
We show that $(v,y)$ is a basis of $R^{q+1}$.
Indeed, let $V \subset R^{q+1}$ be the $R$-span of $v_1,...,v_{q-1}$.
Then $V$ equipped with the symplectic form $\langle\phantom{X},\phantom{Y}\rangle$ is non-degenerate since the Gram matrix of $v_1,...,v_{q-1}$ is invertible.
Therefore, there is a unique $w\in V$ such that $\langle w,v_i\rangle = \langle v_q,v_i\rangle$ for all $i=1,...,q-1$.
%The vector $w$ exists (and is unique) since the Gram matrix of $v_1,...,v_{q-1}$ is invertible.
In the orthogonal decomposition $V \perp V^{\perp}=R^{q+1}$ of $R^{q+1}$, the vectors $v_q-w,y\in V^{\perp}$ are a hyperbolic basis of $V^{\perp}$ since $\Gamma(v_q-w,y)=\psi_2$.
It follows that $(v_1,...,v_{q-1},v_q-w,x)$ is a basis of $R^{q+1}$, hence $(v_1,...,v_{q-1},v_q,x)$ is a basis of $R^{q+1}$.
Similarly, $(u,x)$ is also a basis of $R^{q+1}$.
The $R$-linear endomorphism $B = (v,y)\circ (u,x)^{-1}:R^{q+1} \to R^{q+1}$ sending $(u,x)$ to $(v,y)$ is an isometry and thus has determinant $1$ as $\Sp_{q+1}(R)\subset \SL_{q+1}(R)$.
Since $u,v\in \GL_q(R)$, the matrices $(u,x)$ and $(v,y)$ have the form
$$(u,x) = \left(\begin{array}{c|c}u& \ast \\ \hline 0 & x_0\end{array}\right)\hspace{2ex}\text{ and }\hspace{2ex}(v,y)=\left(\begin{array}{c|c}v& \ast \\ \hline 0 & y_0\end{array}\right).$$
Thus, $x_0\det u = \det(u,x) = \det(v,y) = y_0 \det v$, and the matrix 
$$ \left(\begin{smallmatrix}1_{2r} & 0 & 0  \\ 0 & \det^{-1}  v& 0 \\ 0 & 0 & \det v  \end{smallmatrix}\right) B \left(\begin{smallmatrix}1_{2r} & 0 & 0\\ 0 & \det u & 0\\ 0 & 0 & \det^{-1}  u\end{smallmatrix}\right) \in \Sp_{q+1}(R)$$
has last row equal to $^te_{q+1}=(0,0,...,0,1)$.
In particular, that matrix has the form
$$ \left(\begin{smallmatrix}P& 0 & g  \\ ^th & 1& g_0 \\ 0 & 0 & 1\end{smallmatrix}\right)$$
for some $g,h\in R^{q-1}$, $g_0\in R$, and $P\in \Sp_{q-1}(R)$.
Now we extend the isometry $B$ of $R^{q+1}$ to all of $R^{2n}$ by requiring $Be_i=e_i$ for $i = q+2,...,2n$.
Then $Bu=v$, $B\in \Sp_{2n}(R)$, and for all $M \in \Sp_{2n-q}(R)$ we have $c_{\det v}(M) = B \circ c_{\det u}(M)\circ B^{-1}$ since
$$ \left(\begin{smallmatrix}1_{2r} & 0 & 0 & 0 \\ 0 & \det^{-1}  v& 0  & 0\\ 0 & 0 & \det v & 0 \\ 0 & 0 & 0 & 1_{2n-q-1}\end{smallmatrix}\right) B \left(\begin{smallmatrix}1_{2r} & 0 & 0 &0\\ 0 & \det u & 0 &0\\ 0 & 0 & \det^{-1}  u& 0 \\ 0 & 0 & 0 & 1_{2n-q-1}\end{smallmatrix}\right) = \left(\begin{smallmatrix}P& 0 & g & 0 \\ ^th & 1& g_0  & 0\\ 0 & 0 & 1 & 0 \\ 0 & 0 & 0 & 1_{2n-q-1}\end{smallmatrix}\right).$$
Any such matrix commutes with every matrix in $\Sp_{2n-q}(R)$ because
$$\left(\begin{smallmatrix} P & 0 & g & 0 \\ ^th & 1 & g_0 & 0 \\ 0 & 0 & 1 & 0\\ 0 & 0 & 0 & 1\end{smallmatrix}\right)
\left(\begin{smallmatrix} 1 & 0 & 0 & 0 \\ 0 & 1 & b_0 & ^tb \\ 0 & 0 & 1 & 0 \\ 0 & 0 & a & N\end{smallmatrix}\right)
=
\left(\begin{smallmatrix} P & 0 & g & 0 \\ ^th & 1 & g_0+b_0 & ^tb \\ 0 & 0 & 1 & 0\\ 0 & 0 & a & N\end{smallmatrix}\right)
=
\left(\begin{smallmatrix} 1 & 0 & 0 & 0 \\ 0 & 1 & b_0 & ^tb \\ 0 & 0 & 1 & 0 \\ 0 & 0 & a & N\end{smallmatrix}\right)
\left(\begin{smallmatrix} P & 0 & g & 0 \\ ^th & 1 & g_0 & 0 \\ 0 & 0 & 1 & 0\\ 0 & 0 & 0 & 1\end{smallmatrix}\right)
$$
for all $a,b\in R^{2n-q-1}$, $b_0\in R$, $N\in M_{2n-q-1}(R)$.
This finishes the proof.
\end{proof}

\begin{corollary}
\label{cor:E1IdentWellDef}
For $0 \leq q \leq 2n$, the following map, sending $\alpha \otimes A$ to $f_v(\alpha)$, does not depend on the choice of $v$ and is an isomorphism
\begin{equation}
\label{eqn:E1ShapiroIdentn}
H_p(\Sp_{2n-q}(R)) \otimes_{\Z} \Z[\Skew^+_q(R)] \stackrel{\cong}{\longrightarrow} H_p(\Sp_{2n}(R);\Z[U_q(R^{2n})])
\end{equation}
provided $v\in U_q(R^{2n})$ with $\Gamma(v)=A$ and $v$ generates $R^q$.
\end{corollary}

\begin{proof}
The map does not depend on the choice of $v$, by Lemma \ref{lem:vIndep}.
It is an isomorphism, by Shapiro's isomorphism in view of Lemma \ref{lem:GammaBij}.
\end{proof}

The following lemma identifies the $E^1$-page of the spectral sequence (\ref{eqn:E1spseq}) and its $d^1$ differential.

\begin{lemma}
\label{lem:commdiag}
	For $0 \leq q < 2n$ the following diagram commutes
$$
\xymatrix{
	H_*(\Sp_{2n-q-1}; \Z) \otimes_{\Z} \Z[\Skew_{q+1}^+] \ar[rr]^{\hspace{5ex}\text{(\ref{eqn:E1ShapiroIdentn})}}_{\hspace{5ex}\cong}
	\ar@{->}[d]_{\eps_*\otimes d}
	&& H_*(\Sp_{2n}; \Z[U_{q+1}(R^{2n})])
	\ar@{->}[d]^{(1,d)_*}
	\\
	H_*(\Sp_{2n-q}; \Z)\otimes_{\Z} \Z[\Skew_q^+]   \ar[rr]^{\hspace{7ex}\text{(\ref{eqn:E1ShapiroIdentn})}}_{\hspace{7ex}\cong}
	&&H_*(\Sp_{2n}; \Z[U_q(R^{2n})]).
}
$$
\end{lemma}

\begin{proof}
Recall that $d = \sum_{i=1}^{q+1} (-1)^{i+1}d_i$ where $d_i$ omits the $i$-th entry.
We will show that the diagram commutes with $d_i$ in place of $d$ for $i=1,...,q+1$.
On the component of the upper left corner corresponding to $A\in \Skew^+_{q+1}(R)$ choose $u \in U_{q+1}(R^{2n})$ generating $R^{q+1}$ such that $\Gamma(u)=A$ and $d_iu$ generates $R^{q}$, 
for instance, $(d_iu,u_i)$ in normal form will do; see Lemma \ref{lem:normalForm}.
Then $\Gamma(d_iu) = d_i A$.
In view of Lemma \ref{lem:vIndep} we can use $f_u$ and $f_{d_iu}$ for the horizontal maps.

If $q$ is even then $q+1$ is odd and
going first right then down gives the map 
$$
\renewcommand\arraystretch{2}
\begin{array}{rl}
&(\eps^{2n}_{2n-q-1}\circ c_a, d_iu)_*\\
=& (\eps^{2n}_{2n-q}, d_iu)_* \circ (c_a \circ \eps_{2n-q-1}^{2n-q})= (\eps,d_iu)_*
\end{array}$$
where $c_a$ is conjugation with the diagonal matrix $D$ in $\Sp_{2n-q}(R)$ whose diagonal entries are $(a,a^{-1},1_{2n-q-2})$ and $a=\det u$.
Conjugation with any $D\in \Sp_{2n-q}(R)$ is the identity on $H_*(\Sp_{2n-q}(R);\Z)$.
Thus, this map equals the map obtained by going down then right.

If $q$ is odd, then $q+1$ is even and  going right then down is $(\eps, d_iu)$ whereas going down then right is 
$(\eps^{2n}_{2n-q}\circ c_a \circ \eps^{2n-q}_{2n-q-1}, d_iu) = (\eps^{2n}_{2n-q-1}, d_iu) $ since $\ c_a \eps^{2n-q}_{2n-q-1}= \eps^{2n-q}_{2n-q-1}$ where $c_a$ is conjugation with the diagonal matrix
$(a,a^{-1},1_{2n-q-1})$ of $\Sp_{2n-q+1}(R)$ and $a=\det d_iu$.
\end{proof}

\section{The Limit Theorem}
\label{sec:LimitThm}

The goal of this section is to prove the Limit Theorem \ref{thm:LimitThm} which is fundamental in our proof of degeneration of the spectral sequence (\ref{eqn:E1spseq}) in Section \ref{sec:E2degeneration}.
\vspace{1ex}

Let $R$ be a commutative ring (which, for now, need not be local).
An $R$-module $M$ carries a left action $R \times M \to M: (a,x) \mapsto ax$ of the multiplicative monoid $(R,\cdot,1)$ of $R$ which is linear in $M$.
In particular, it is a module over the associated integral monoid ring $\Z[R]=\Z[R,\cdot,1]$.
We denote by $\langle a \rangle$ the element of $\Z[R]$ corresponding to $a\in R$
and note that $\Z\langle 0 \rangle \subset \Z[R]$ is an ideal.
Since $0\cdot M = 0$, the $R$-module $M$ is naturally a module over the quotient ring 
$$\Z_0[R]=\Z[R]/\Z\langle 0 \rangle = \Z[R,\cdot,1]/\Z\langle 0 \rangle.$$
By functoriality, the multiplicative action of $R$ on $M$ induces a multiplicative action on $H_q(M)$, $M^{\otimes_{\Z} q}$, $\Lambda_{\Z}^qM$, and $M(q)$ where the latter is $M$ with action through the $q$-th power of its natural action. 
For $q\geq 1$, all those modules are therefore $\Z_0[R]$-modules. 
For instance, for $q\geq 1$, the $\Z_0[R]$-module structure on 
\begin{equation}
\label{eqn:M(q)dfn}
M(q)\hspace{6ex} \text{ is } \hspace{6ex} \left(\sum_{i=1}^rn_i\langle a_i\rangle \right) \cdot x = \sum_{i=1}^rn_i a^q_i \, x.
\end{equation}

A $\Z_0[R]$-module $M$ is an $R$-module if and only if the multiplicative left action of $R$ on $M$ is also linear in $R$, that is, if for all $a,b\in R$, the element $\langle a\rangle + \langle b \rangle - \langle a+b \rangle$ acts as zero on $M$.
We may call such $\Z_0[R]$-modules {\em linear}.
The criterion for linearity is the $m=2$-case of the following generalisation.
For a sequence $x=(x_1,....,x_m)$ of $m$ elements in $R$, and subset $J \subset \{1,...,m\}$ we denote by $x_J$ the partial sum
$$x_J = \sum_{j\in J} x_j \in R.$$
Then a $\Z_0[R]$-module $M$ is linear if and only if the element
$$-\sum_{\emptyset \neq J \subset \{1,...,m\}} (-1)^{|J|}\ \langle x_J\rangle \hspace{3ex}\in \hspace{2ex} \Z_0[R]$$
acts as zero on $M$ for all $m\geq 2$ and all sequences $x=(x_1,....,x_m)$ of $m$ elements in $R$.
More generally, we have the following.
Our convention is that $x^0=1$ for $x\in R$ even if $x=0$.

\begin{lemma}
\label{lem:MultLin}
Let $R$ be a ring, $M$ an abelian group and let $t\geq 1$ be an integer.
Let 
$$[ \phantom{x,...,y}]:R^{\times t} \to M: (a_1,...,a_t) \mapsto [a_1,...,a_t]$$ be a $\Z$-multilinear map.
Let $x=(x_1,...,x_m)$ be a sequence of $m\geq 1$ elements in $R$.
Let $p_1(X),...,p_t(X) \in R[X]$ be polynomials of degrees $\gamma_1,...,\gamma_t\geq 0$ with $\gamma_1+\cdots + \gamma_t <m$.
Then 
\begin{equation}
\label{eqn:lem:MultLin}
-\sum_{\emptyset \neq J \subset \{1,...,m\}} (-1)^{|J|}\ [p_1(x_J), \cdots, p_t(x_J)] =
 [p_1(0), \dots , p_t(0)].
 \end{equation}
\end{lemma}
%\edit{convention that $x^0=1$ for $x\in R$ even if $x=0$}

\begin{proof}
We first prove the lemma for $p_i(X)=a_iX^{\gamma_i}$, $a_i\in R$.
If $\gamma_1=\cdots =\gamma_t=0$ then the left term in (\ref{eqn:lem:MultLin}) is $[a_1,...,a_t]=[p_1(0), \dots , p_t(0)]$ because 
$$1+\displaystyle{\sum_{\emptyset \neq J \subset \{1,...,m\}} (-1)^{|J|}} = \displaystyle{\sum_{J \subset \{1,...,m\}} (-1)^{|J|}} = (1-1)^m =0$$
for $m\geq 1$.

If $\gamma_1+ \cdots + \gamma_t\geq 1$ 
we write $[n]$ for the set $\{1,...,n\}$. 
Then the left term in (\ref{eqn:lem:MultLin}) is
$$\renewcommand\arraystretch{2.5}
\begin{array}{rcl}
&&\displaystyle{\sum_{\emptyset \neq J \subset \{1,...,m\}} (-1)^{|J|}\ [a_1(x_J)^{\gamma_1}, \cdots, a_t(x_J)^{\gamma_t}]}\\
&=&
\displaystyle{\sum_{\emptyset \neq J \subset [m], \sigma_i:[\gamma_i] \to J, 1 \leq i \leq t} (-1)^{|J|}}\ 
[a_1 x_{\sigma_1(1)}\cdots x_{\sigma_1(\gamma_1)}, \dots, a_t x_{\sigma_t(1)}\cdots x_{\sigma_t(\gamma_t)]}\\
&=&
\displaystyle{\sum_{\sigma_i:[\gamma_i] \to [m], 1 \leq i \leq t}  [a_1 x_{\sigma_1(1)}\cdots x_{\sigma_1(\gamma_1)}, \dots ,a_t x_{\sigma_t(1)}\cdots x_{\sigma_t(\gamma_t)}]\sum_{\bigcup_{i=1}^{t} \im(\gamma_i)\subset J \subset [m]}  (-1)^{|J|}}\\
&=& 0
\end{array}
$$
since $\emptyset \neq \bigcup_{i=1}^{t} \im(\gamma_i) \subsetneq [m]$ as $1 \leq \gamma_1+ \cdots + \gamma_t<m$, and
for $S \subsetneq [m]$ we have
$$\sum_{S\subset J \subset [m]}  (-1)^{|J|} =(-1)^{|S|} \sum_{J \subset [m]-S}(-1)^{|J|} = (-1)^{|S|}(1-1)^{m-|S|} =0.$$

Now we assume that $p_1(X),...,p_t(X) \in R[X]$ are arbitrary polynomials of degrees $\gamma_1,...,\gamma_t\geq 0$ with $\gamma_1+\cdots + \gamma_t <m$.
Each polynomial $p(X)$ is the sum of $p(0)$ and a $\Z$-linear combination of polyomials $a_{\gamma}X^{\gamma}$ with $\gamma \geq 1$ and $a_{\gamma}\in R$.
It follows that the left term of (\ref{eqn:lem:MultLin}) is the sum of
\begin{equation}
\label{eqn:ZLinCombInv2}
-\displaystyle{\sum_{\emptyset \neq J \subset \{1,...,m\}} (-1)^{|J|} \ [p_1(0),\dots, p_t(0)]}
\end{equation}
and a $\Z$-linear combination of terms
\begin{equation}
\label{eqn:ZLinCombInv}
-\sum_{\emptyset \neq J \subset \{1,...,m\}} (-1)^{|J|}\ [a_1(x_J)^{\delta_1}, \cdots, a_t(x_J)^{\delta_t}]
\end{equation}
for some $a_i\in R$ and where $0 \leq \delta_i $ and $1 \leq \delta_1+\cdots +\delta_t<m$.
By the first part of the proof, the terms (\ref{eqn:ZLinCombInv}) are zero and therefore, the left term of (\ref{eqn:lem:MultLin}) equals (\ref{eqn:ZLinCombInv2}) which is
$[p_1(0), \dots p_t(0)]$, again by the first part of the proof.
\end{proof}

For a sequence $a=(a_1,...,a_m)$ of $m$ elements in $R$ and a polynomial $p(X) \in R[X]$ with coefficients in $R$, we write $s_p(a) \in \Z[R]$ for the element 
$$s_p(a) = -\sum_{\emptyset \neq J \subset \{1,...,m\}} (-1)^{|J|}\ \langle p(x_J)\rangle \hspace{3ex}\in \hspace{2ex} \Z[R].$$

\begin{remark}
For $p(X)=X$, the element $s_p(a)$ was first considered in \cite{myEuler} to prove optimal homology stability for special linear groups.
Note that for $m\geq 1$
$$s_1 = -\sum_{\emptyset \neq J \subset \{1,...,m\}} (-1)^{|J|} = 1.$$
\end{remark}

\begin{definition}
\label{dfn:quasiLinear}
Let $R$ be a commutative ring. 
A $\Z_0[R]$-module $M$ is called {\em quasi-linear} if for every polynomial $p\in R[X]$ there is an integer $m_0\geq 0$ such that for all integers $m\geq m_0$ and all sequences $a=(a_1,...,a_m)$ of $m$ elements in $R$, we have
$\sigma^{-1}M=0$ where $\sigma = s_p(a)-\langle p(0)\rangle \in \Z_0[R]$.
\end{definition}

Note that the category of quasi-linear $\Z_0[R]$-modules is a Serre abelian subcategory of the abelian category of all $\Z_0[R]$-modules, that is, subobjects, quotients and extensions of quasi-linear $\Z_0[R]$-modules in the category of $\Z_0[R]$-modules are quasi-linear.

\begin{example}
By Lemma \ref{lem:MultLin}, for all $R$-modules $M$ the $\Z_0[R]$-modules
$M^{\otimes_{\Z} q}$, $\Lambda_{\Z}^qM$, and $M(q)$ are quasi-linear for all $q\geq 1$.
We will see in Proposition \ref{prop:HpLambdaM3} below that $H_s(M(q))$, $s\geq 1$, and the relative integral homology groups $H_s(M(q) \rtimes G,G)$ are quasi-linear as well if $G$ acts on $M$ by means of $R$-module homomorphisms; see Example \ref{ex:MqSemiDirect}.
\end{example}

\begin{remark}
Let $(R,\mathscr{M})$ be a local ring with infinite residue field $R/\mathscr{M}$, and consider the ring homomorphism $\Z_0[R] \to \Z$ sending $R^*$ to $1$ and $\mathscr{M}$ to $0$.
This makes $M=\Z$ into a $\Z_0[R]$-module which is not quasi-linear, in particular, $\Z_0[R]$ is not quasi-linear.
Indeed, if $p(X)=X$ then $s_p(a)$ acts as $1$ on $\Z$ for all sequences $a=(a_1,...,a_m)$ of units $a_i\in R^*$ such that $a_J\in R^*$ for all $\emptyset \neq J \subset \{1...,m\}$, and $\langle p(0)\rangle =0$ acts as $0$.
In particular $\sigma^{-1}M=M$ for all $\sigma=s_p(a)-\langle p(0)\rangle$ and all sequences 
$a=(a_1,...,a_m)$ of units in $R$ as above.
Since $R$ has infinite residue field, $m$ can be chosen as large as we want. 
\end{remark}

In order to state our Limit Theorem \ref{thm:LimitThm} we need to introduce some terminology.

\begin{definition}
Let $R$ be a local ring with infinite residue field $k$ and denote by $\pi:R \to k$ the quotient map.
A subset $\D \subset R$ of elements in $R$ is called {\em region} if $\D=\pi^{-1}\pi(\D)$.
A region $\D \subset R$ is called {\em dense} if $k-\pi(\D)$ is finite.
\end{definition}

\begin{definition}
\label{dfn:f(t)}
Let $R$ be a local ring, and $\D \subset R$ a dense region of $R$.
A function $f:\D \to \Z_0[R]$ is called {\em admissible}
if there are polynomials $P \in \Z[X_1,...,X_n]$, $P_i,Q_i \in R[X]$, $i=1,...,n$, such that $Q_i(t)\in R^*$ for all $t\in \D$ and 
\begin{equation}
\label{eqn:f(t)dfn}
f(t) = P\left(\left\langle \frac{P_1(t)}{Q_1(t)}\right\rangle, \dots, \left\langle \frac{P_n(t)}{Q_n(t)}\right\rangle\right) \in \Z_0[R]
\end{equation} 
for all $t\in \D$.
The polynomials $P,P_i,Q_i$ are called {\em presentation of $f$}.

For $a \in R$, we say that $f$ is {\em defined at $a$} (relative to the presentation $(P,P_i,Q_i)$) if the elements $Q_i(a)\in R$ are units in $R$.
Clearly, $f$ is defined at all elements of $\D$.
Note that if $f$ is defined at $a\in R$ then $f(a)$ is a well-defined element in $\Z_0[R]$, given by (\ref{eqn:f(t)dfn}), though the value $f(a)$ may depend on the presentation of $f$.
\end{definition}

\begin{definition}
\label{dfn:Lim}
Let $\D \subset R$ be a region of a local ring $R$, and 
let $f:\D \to \Z_0[R]$ be an adimissible function represented by $(P,P_i,Q_i)$ as in (\ref{eqn:f(t)dfn}).
For $a \in R\cup \{\infty\}$ we say that the {\em limit $\lim_{t\to a}f(t)$ of $f$ when $t$ tends to $a$} exists and write $$\lim_{t\to a}f(t)=L\hspace{2ex}\in \hspace{2ex}\Z_0[R]$$ if either of the following holds.
\begin{enumerate}
\item
If $a\in R$, then we require $f$ to be defined at $a$ and set $L = f(a)$.
\item
\label{dfn:Lim:item2}
If $a=\infty$, then we require $\deg P_i \leq \deg Q_i$ and the coefficients of the highest degree monomials of $Q_i(X)$ to be units, $i=1,...,n$.
Then 
$$\bar{Q}_i(X) = X^{\deg Q_i}Q_i(1/X),\hspace{3ex}\bar{P}_i(X) = X^{\deg Q_i}P_i(1/X)$$ are polynomials with $\bar{Q}_i(0) \in R^*$, $i=1,...,n$.
We note that
$$f(1/t) = \bar{f}(t) = P\left(\left\langle \frac{\bar{P}_1(t)}{\bar{Q}_1(t)}\right\rangle, \dots, \left\langle \frac{\bar{P}_n(t)}{\bar{Q}_n(t)}\right\rangle\right) \hspace{2ex} \in \hspace{2ex}  \Z_0[R]$$
for $1/t\in \D$, and that $\bar{f}$ is defined at $0$ relative to the presentation $(P,\bar{P}_i,\bar{Q}_i)$.
We set
$$L=\lim_{t\to \infty}f(t) = \lim_{t\to 0}f(1/t) = \lim_{t\to 0}\bar{f}(t) = \bar{f}(0).$$
\end{enumerate}
\end{definition}

We do not know if $\lim_t f(t)$ does or does not depend on the presentation of $f$.
For the purpose of this paper, the limit will always be calculated relative to a given presentation of $f$.

\begin{theorem}[Limit Theorem]
\label{thm:LimitThm}
Let $R$ be a local ring with infinite residue field.
Let $M$ be a quasi-linear $\Z_0[R]$-module, and let $x\in M$.
Let $\D \subset R$ be a dense region of $R$, and
let $f:\D \to \Z_0[R]$ be an admissible function with given presentation.
Assume that $f(t) \in \sqrt{\ann(x)} \subset \Z_0[R]$ for all $t\in \D$.
Then for all $a\in R \cup \{\infty\}$, 
if $\lim_{t \to a}f(t)$ exists in $\Z_0[R]$ in the given presentation then that limit satisfies
$$\lim_{t \to a}f(t) \in \sqrt{\ann(x)}.$$
\end{theorem}

\begin{remark}
The Limit Theorem does not hold for all $\Z_0[R]$-modules $M$.
For instance, let $K$ be an infinite field, and consider the ring homomorphism
$\Z_0[K] \to \Z$ sending the elements of $K^*$ to $1$ (and $\langle 0 \rangle$ to $0$).
This makes the target $M=\Z$ into a $\Z_0[K]$-module.
For $f(t)=-\langle t \rangle + 1$, presented by $P(X)=-X+1$ and $P_1(X)=X$, $Q_1(X)=1$, we have $f(t)M=0$ for all $t\in \D=K^*$, but $f(0)=1$ is not in the radical of the annihilator of a generator of $M$.
%, otherwise $f(0)^n=1$ would annihilate $M$.
Therefore, some condition such as "quasi-linear" is required for the theorem to hold.
\end{remark}

\begin{proof}[Proof of Theorem \ref{thm:LimitThm}]
Let $(P,P_i,Q_i)$ be the given presentation of $f$ as in (\ref{eqn:f(t)dfn}).
We first consider the case $a=0$.
Since $\lim_{t\to 0}f(t)$ exists, we have $Q_i(0)\in R^*$ for all $i=1,...,n$.
Let $d_i$ be the highest power of $X_i$ occurring in $P(X_1,...,X_n)$.
Then 
$g(t)= \langle Q_1(t)^{d_1} \cdots Q_n(t)^{d_n}\rangle \ f(t)$ 
is an integer linear combination of expressions $\langle p_j(t) \rangle$ with $p_j(X)\in R[X]$ polynomials, $j=1,...,\ell$, for some $\ell \in \N$:
$$g(t) =\langle Q_i(t)^{d_i} \cdots Q_i(t)^{d_i}\rangle f(t)= \sum_{j=1}^{\ell}n_j\langle p_j(t) \rangle.$$
Since $M$ is a quasi-linear $\Z_0[R]$-module, we can choose an integer $m_0$ such that $\sigma_j(a) = s_{p_j}(a) -\langle p_j(0)\rangle$ satisfies $\sigma_j(a)^{-1}M=0$ for all $j=1,...,\ell$ and all sequences $a=(a_1,...,a_m)$ of $m$ elements in $R$ with $m\geq m_0$.
In particular, $\sigma_j(a)\in \sqrt{\ann(x)}$ and hence 
\begin{equation}
\label{eqn:LimThm1}
s_g(a)-g(0) = \sum_{j=1}^{\ell}n_j\sigma_j (a)\in \sqrt{\ann(x)}
\end{equation}
where (abusing notation slightly)
$$s_g(a) := -\sum_{\emptyset \neq J \subset \{1,...,m\}} (-1)^{|J|}\ g(a_J) =  \sum_{j=1}^{\ell}n_js_{p_j}(a).$$

Fix $m\geq m_0$ and choose a sequence $a=(a_1,...,a_m)$ of $m$ elements in $R$ such that $a_J\in \D$ for all $\emptyset \neq J \subset \{1,...,m\}$.
This is possible for if we denote by $\pi:R \to k$ the quotient map to the residue field of $R$, and if we
have chosen $(a_1,...,a_t)$ such that $a_J\in \D$ for all $\emptyset \neq J \subset \{1,...,t\}$, then $a_{t+1}\in R$ can be any element such that $\pi({a}_{t+1})$ is not the solution $x\in k$ to one of the finitely many non-trivial linear equations $x+\pi({a}_J) = y$, $y\in k-\pi(\D)$, $J\subset \{1,...,t\}$.
Such $x\in k$ exists since $k$ is infinite.
Since $a_J\in \D$, we have $f(a_J)\in \sqrt{\ann(x)}$ for all $\emptyset \neq J \subset \{1,...,m\}$, by assumption.
Then $g(a_J)\in \sqrt{\ann(x)}$ for all $\emptyset \neq J \subset \{1,...,m\}$. 
As a $\Z$-linear combination of the $g(a_J)$'s we then have $s_g(a)\in \sqrt{\ann(x)}$.
By (\ref{eqn:LimThm1}), we have $g(0)\in \sqrt{\ann(x)}$ and thus,
$$\lim_{t\to 0}f(t) = f(0) = \langle Q_1(0)^{-d_1} \cdots Q_n(0)^{-d_n}\rangle\ g(0) \in \sqrt{\ann(x)}$$
since $Q_i(0)\in R^*$, $i=1,...,n$.

Now assume $a\in R$ arbitrary.
Define $\bar{P}_i(X) = P_i(X+a)$, $\bar{Q}_i(X) = Q_i(X+a)$, $\bar{f}(t) = f(t+a)$, $\bar{P} = P$,  $\bar{\D}=\D-a$.
Then $\bar{f}(t)\in \sqrt{\ann(x)}$ for all $t\in \bar{\D}$, and the case of $t\to 0$ treated above shows that $\lim_{t\to a}f(t) = \lim_{t\to 0}\bar{f}(t)\in \sqrt{\ann(x)}$.

Finally assume $a=\infty$.
Set $\bar{P}_i(X) $, $\bar{Q}_i(X)$, $\bar{f}(t)=f(1/t)$ as in Definition \ref{dfn:Lim} (\ref{dfn:Lim:item2}).
Note that $\bar{\D} = \{t \in R^*|\ t^{-1}\in \D\}$ is a dense region of $R$ since $\D$ is.
Then $\bar{f}(t)=f(1/t) \in \sqrt{\ann(x)}$ for $t\in \bar{\D}$.
By the case $a=0$ treated above, we have 
$$\lim_{t \to \infty}f(t) = \lim_{t \to 0} \bar{f}(t) \in \sqrt{\ann(x)}.$$
\end{proof}

\begin{remark}
\label{rmk:qlinTrivAction}
Let $R$ be a local ring with infinite residue field.
If the induced action of $R^*$ on a quasi-linear $\Z_0[R]$-module $M$ is trivial then $M=0$.
Indeed, the admissible function $f:\D=R^* \to \Z_0[R]$ defined by
$f(t)=-\langle t \rangle +1$ has $f(t)M=0$ for all $t\in \D$ but $\lim_{t\to 0}f(t)=1$ is in $\sqrt{\ann(x)}$ for $x\in M$ if and only if $x=0$.
By the Limit Theorem \ref{thm:LimitThm} we must have $M=0$.
\end{remark}

\section{Quasi-linear modules and group homology}

The goal of this section is to prove in Proposition \ref{prop:HpLambdaM3} below that the relative homology groups $H_s(G,K)$ are quasi-linear for certain $(R,\cdot,1)$-equivariant inclusions of groups $K \subset G$.
This will be applied to show that the relative homology groups $H_s(\Sp_{2r+1}(R),\Sp_{2r}(R))$ are quasi-linear $\Z_0[R]$-modules. 
At the end of the section we will give a few first applications of the Limit Theorem \ref{thm:LimitThm}.

\vspace{1ex}
 
For an integer $t\geq 1$, we consider the ring homomorphism
$$\ffi_t:\Z[R,\cdot,1] \to R^{\otimes t}: [a] \mapsto a^{\otimes t} = a \otimes \cdots \otimes a$$
where $a\in R$.
Assume the multiplicative monoid $(R,\cdot,1)$ of $R$ acts on a group $G$ from the left through group homomorphisms. 
By functoriality, $(R,\cdot,1)$ acts on the homology group $H_q(G)$ from the left through abelian group homomorphisms, that is, $H_q(G)$ is a left $\Z[R]$-module.
Recall from (\ref{eqn:M(q)dfn}) the $\Z_0[R]$-module $M(q)$ associated with an $R$-module $M$ and an integer $q\geq 0$.

\begin{lemma}
\label{lem:sHNZero}
Let $R$ be a commutative ring whose underlying abelian group $(R,+,0)$ is torsion free.
Let $A$, $B$ be $R$-modules, and let $r,\alpha,\beta \geq 1$ be integers.
Let 
\begin{equation}
\label{eqn:lem:sHNZero}
1 \to  B(\beta)\to N \to A(\alpha)  \to 1
\end{equation}
be a $(R,\cdot,1)$-equivariant central extension of groups.
Let $\sigma\in \Z[R]$ be such that $\ffi_t(\sigma)=0$ for $1 \leq t \leq r$.
Then
$\sigma^{-1}H_s(N)=0$ whenever $1 \leq s\cdot \max(\alpha,\beta) \leq r$.
\end{lemma}

Note that the group $N$ in the lemma need not be abelian.

\begin{proof}[Proof of Lemma \ref{lem:sHNZero}]
We will first prove the lemma when $A$ is torsion-free as abelian group.
To do so we show that in this case
\begin{equation}
\label{eqn:lem:Proof:sHNZero1}
\sigma^{-1}\left(H_p(A_{\alpha})\otimes H_q(B_{\beta})\right) =0 \hspace{5ex}\text{for}\hspace{5ex}1 \leq \alpha p + \beta q \leq r,
\end{equation}
and then apply the Hochschild-Serre spectral sequence to (\ref{eqn:lem:sHNZero}).
To prove (\ref{eqn:lem:Proof:sHNZero1}) we first also assume that $B$ is torsion-free as abelian group.
Then $H_p(A)\otimes H_q(B) = \Lambda_{\Z}^p(A) \otimes \Lambda_{\Z}^q(B)$, functorial in $A$ and $B$.
In particular, the result of the action of $a\in R \subset \Z[R]$ on $(x_1\wedge \cdots \wedge x_p) \otimes (y_1\wedge \cdots \wedge y_q) \in H_p(A({\alpha}))\otimes H_q(B({\beta}))$ is 
$(a^{\alpha}x_1\wedge \cdots \wedge a^{\alpha}x_p) \otimes (a^{\beta}y_1\wedge \cdots \wedge a^{\beta}y_q)$.
This is the image of $\ffi_{\alpha p + \beta q}(a)$ under the $\Z$-linear map
\begin{equation}
\label{eqn:lem:Proof:sHNZero2}
R^{\otimes \alpha p} \otimes R^{\otimes \beta q} \longrightarrow \Lambda_{\Z}^p(A) \otimes \Lambda_{\Z}^q(B)
\end{equation}
which uniquely extends the $\Z$-multilinear map
$$R^{\alpha p} \times R^{\beta q} \longrightarrow \Lambda_{\Z}^p(A) \otimes \Lambda_{\Z}^q(B)$$
sending $(M, N) \in R^{\alpha p} \times R^{\beta q} = M_{\alpha,p}(R)\times M_{\beta,q}(R)$ to 
$$\left((\prod_{i=1}^{\alpha}M_{i,1}) x_1\wedge \cdots \wedge (\prod_{i=1}^{\alpha}M_{i,p}) x_p\right) \otimes \left((\prod_{i=1}^{\beta} N_{i,1}) y_1 \wedge \cdots \wedge (\prod_{i=1}^{\beta} N_{i,q})y_q\right).$$
In particular, the result of the action of $\sigma \in \Z[R]$ on $(x_1\wedge \cdots \wedge x_p) \otimes (y_1\wedge \cdots \wedge y_q)$ is the image of $\ffi_{\alpha p + \beta q}(\sigma)$ under the $\Z$-linear map
(\ref{eqn:lem:Proof:sHNZero2}). 
But $\ffi_{t}(\sigma)=0$ for $1 \leq t \leq r$.
Hence, $\sigma( \Lambda_{\Z}^p(A) \otimes \Lambda_{\Z}^q(B)) =0$ for $1 \leq \alpha p + \beta q \leq r$.
In particular, (\ref{eqn:lem:Proof:sHNZero1}) holds when $A$ and $B$ are torsion-free.

Now we prove (\ref{eqn:lem:Proof:sHNZero1}) when $A$ is torsion-free as abelian group and $B$ is an arbitrary $R$-module.
Choose a surjective weak equivalence of simplicial $R$-modules $B_* \to B$ with $B_i$ a projective $R$-module for all $i\in \N$.
For instance, the simplicial $R$-module corresponding to an $R$-projective resolution of $B$ under the Dold-Kan correspondence will do.
Each $B_i$ is a torsion free abelian group since $R$ is.
The classifying space functor $\B$ induces an $(R,\cdot,1)$-equivariant weak equivalence of simplicial sets
$\B(B_*(\beta)) \to \B B(\beta)$.
Tensoring the spectral sequence of the simplicial space $n \mapsto \B B_n$,
$$E^1_{s,t} = H_t(\B B_s) \Rightarrow H_{s+t}(\B B_*) = H_{s+t}(\B B)= H_{s+t}(B),$$
with the flat $\Z$-module $H_p(A)=\Lambda_{\Z}^pA$ yields the spectral sequence of $\Z[R]$-modules
$$H_p(A(\alpha))\otimes E^1_{s,t} = H_p(A(\alpha))\otimes H_t(B_s(\beta)) \Rightarrow H_p(A(\alpha)) \otimes H_{s+t}(B(\beta)).$$
Localising at $\sigma$, this yields a spectral sequence with trivial $E^1_{s,t}$-term for $1\leq \alpha p +\beta t \leq r$.
Since $t \leq s+t$ for $0 \leq s,t$, the $E^1_{s,t}$-term of the localised spectral sequence is trivial for $1\leq \alpha p +\beta (s+t) \leq r$ (and $p,s,t\geq 0$).
This proves (\ref{eqn:lem:Proof:sHNZero1}) when $A$ is torsion-free as abelian group.

Now we prove the lemma when $A$ is torsion-free as abelian group.
In this case, the integral homology groups $H_*(A)=\Lambda_{\Z}^*(A)$ are torsion free and the natural map
$H_p(A) \otimes F \to H_p(A;F)$ is an isomorphism for any abelian group $F$, by the Universal Coefficient Theorem.
Since the extension (\ref{eqn:lem:sHNZero}) is central, the group $A$ acts trivially on $H_*(B)$ and the Hochschild-Serre spectral sequence of the group extension has the form
$$E^2_{p,q}= H_p(A;H_q(B)) \cong H_p(A)\otimes H_q(B) \Rightarrow H_{p+q}(N).$$
The spectral sequence is functorial in the exact sequence (\ref{eqn:lem:sHNZero}).
In particular, it is equivariant for the $(R,\cdot,1)$-action and thus a spectral sequence of $\Z[R]$-modules.
Localising the spectral sequence at $\sigma$ yields a spectral sequence with $E^2$-term
$\sigma^{-1}E^2_{p,q}=0$ for $1 \leq \alpha p+\beta q \leq r$, by (\ref{eqn:lem:Proof:sHNZero1}).
This implies the lemma in case $A$ is torsion-free.

Finally, we prove the lemma for arbitrary $R$-modules $A$ and $B$.
As above, we choose a surjective weak equivalence $A_* \to A$ of simplicial $R$-modules with $A_n$ a projective $R$-module for all $n$.
%For instance, the simplicial $R$-module corresponding to an $R$-projective resolution of $A$ under the Dold-Kan correspondence will do.
Then each $A_n$ is flat as abelian group since $R$ is.
Let $N_n = N\times_{A(\alpha)}A_n(\alpha)$.
The action of $(R,\cdot,1)$ on $N$, $A(\alpha)$, and $A_n(\alpha)$ defines an action of $(R,\cdot,1)$ on $N_n$.
We obtain a simplicial  $(R,\cdot,1)$-equivariant central extension
$$1 \to B(\beta) \to N_{\ast} \to A_{\ast}(\alpha) \to 1$$
with degree-wise torsion-free base $A_n$.
The surjection $N_* \to N$ of simplicial groups has contractible kernel as it equals the kernel of the surjective weak equivalence $A_{\ast} \to A$.
In particular, the map on classifying spaces $\B |s\mapsto N_s| = |s \mapsto \B N_s|  \to \B N$ is an $(R,\cdot ,1)$-equivariant weak equivalence.
%For each $n$ we have an $(R,\cdot ,1)$-equivariant central extension $B \to N_n \to A_n$ with torsion-free base $A_n$. 
By the torsion free case treated above, we have $\sigma^{-1}H_q(\B N_s) = 0$ for $1\leq q\cdot \max(\alpha,\beta) \leq r$ and for all $s\geq 0$.
Therefore, the spectral sequence of the simplicial space $s \mapsto \B N_s$,
$$E^2_{p,q}=\pi_p|s \mapsto H_q(\B N_s)| \Rightarrow H_{p+q}(\B N_*) = H_{p+q}(\B N),$$
localised at $\sigma$ has trivial $E^2_{p,q}$-term for $1\leq q\cdot \max(\alpha,\beta) \leq r$ and for all $p$.
In particular, $\sigma^{-1}E^2_{p,q}=0$ whenever $1\leq (p+q)\cdot \max(\alpha,\beta) \leq r$ (and $0\leq p,q$).
This proves the lemma.
\end{proof}

\begin{lemma}
\label{lem:ffitpaZero}
Let $a=(a_1,...,a_m)$ be a sequence of $m$ elements in $R$, and let $p(X) \in R[X]$ be a polynomial of degree $d$ with coefficients in $R$.
Then $s_p(a) - \langle p(0)\rangle \in \Z[R]$ is sent to zero under the map $\ffi_t$  for $1 \leq td < m$:
$$\ffi_t\left( s_p(a) - \langle p(0)\rangle \right) = 0 \in R^{\otimes t}.$$
\end{lemma}

\begin{proof}
The image of $s_p(a)$ in $R^{\otimes t}$ is
$$-\sum_{\emptyset \neq J \subset \{1,...,m\}} (-1)^{|J|}\ P(x_J)^{\otimes t}.$$
We apply Lemma \ref{lem:MultLin} to the canonical $\Z$-multilinear map
$R^{\times t} \to R^{\otimes t}:(x_1,...,x_t) \mapsto [x_1,...,x_t] = x_1 \otimes \cdots \otimes x_t$ and find that
$$\renewcommand\arraystretch{2.5}
\begin{array}{rcl}
\ffi_t(s_p(a)) &= &-\sum_{\emptyset \neq J \subset \{1,...,m\}} (-1)^{|J|}\ [p(a_J), \cdots, p(a_J)]\\
& = &
 [p(0), \dots , p(0)] =p(0)^{\otimes t} = \ffi_t(\langle p(0)\rangle ).
 \end{array}$$
 \end{proof}

\begin{proposition}
\label{prop:HpLambdaM3}
Let $R$ be a commutative ring, let $A$, $B$ be $R$-modules, and let $\alpha,\beta \geq 1$ be integers.
Let $G$, $K$, $N$ be groups with left $(R,\cdot,1)$-actions which are part of $(R,\cdot,1)$-equivariant exact sequences of groups
$$1 \to  B(\beta)\to N \to A(\alpha) \to 1, \hspace{10ex} 1 \to N \to G \stackrel{\rho}{\to} K \to 1$$
in which the first sequence is a central extension, the second sequence has an $(R,\cdot,1)$-equivariant splitting $i:K \to G$ such that $\langle 0 \rangle:G \to G$ is $i\circ \rho$, and the action of $(R,\cdot,1)$ on $K$ is trivial.
Then for all $s\in \Z$ the relative homology groups $H_s(G,K)$ are quasi-linear $\Z_0[R]$-modules where $K$ is considered a subgroup of $G$ by means of the inclusion $i$.
\end{proposition}

\begin{proof}
The action of $\langle 0 \rangle$ on $H_s(G,K)$ factors through $H_s(K,K)=0$.
Hence, the $\Z[R]$-module $H_s(G,K)$ is a $\Z_0[R]$-module.
We will prove that for every sequence $a=(a_1,...,a_m)$ of $m$ elements in $R$ and every polynomial $p(X) \in R[X]$ of degree $d$ with coefficients in $R$, the element $\sigma = s_p(a)-\langle p(0)\rangle \in \Z[R]$ satisfies 
\begin{equation}
\label{eqn:Prop:qlin}
\sigma^{-1}H_s(G,K)=0,\hspace{5ex} \text{provided}\hspace{2ex}sd\max(\alpha,\beta)<m.
\end{equation}
This establishes that $H_s(G,K)$ is quasi-linear with $m_0=sd\max(\alpha,\beta)$ in Definition \ref{dfn:quasiLinear}.

To prove (\ref{eqn:Prop:qlin}), assume first that the underlying abelian group $(R,+,0)$ of $R$ is torsion-free.
By functoriality, the Hochschild-Serre spectral sequence 
\begin{equation}
\label{eqn:prop:HpLambdaM2}
E^2_{p,q}=H_p(K;H_q(N)) \Rightarrow H_{p+q}(G)
\end{equation}
of the extension $1 \to N \to G \to K\to 1$ carries an action of the monoid $(R,\cdot,1)$ induced from the action of that monoid on the extension. 
Section $i:K\to G$ and projection $\rho:G \to K$ make the extension $1 \to 1 \to K \to K \to 1$ of groups with (trivial) $(R,\cdot,1)$-action a direct factor of $1 \to N \to G \to K\to 1$, hence its Hochschild-Serre spectral sequence of (trivial) $\Z[R]$-modules (which degenerates at $E^2$) is a direct factor of that of (\ref{eqn:prop:HpLambdaM2}).
Its complement yields the strongly convergent spectral sequence 
$$\tilde{E}^2_{p,q}=H_p(K;\tilde{H}_q(N)) \Rightarrow H_{p+q}(G,K)$$
where $\tilde{H}_q(N)=H_q(N)$ for $q\geq 1$ and $0$ otherwise.
The action of $g\in K$ on $H_q(N)$ is induced by conjugation with $i(g)$ on $N$.
Since $(R,\cdot,1)$ acts trivially on $K$ and $i$ is equivariant, the action of $(R,\cdot,1)$ on $N$ and the action of $K$ on $N$ commute.
It follows that $\sigma^{-1}\tilde{E}^2_{p,q}=\sigma^{-1}H_p(K;\tilde{H}_q(N)) = H_p(K;\sigma^{-1}\tilde{H}_q(N)) = 0$ for $0 \leq q \max(\alpha,\beta) \leq r$ and any $p$, by Lemmas \ref{lem:sHNZero} and \ref{lem:ffitpaZero}.
Hence, $\sigma^{-1}H_s(G,K)=0$ for $0 \leq s\cdot \max(\alpha,\beta) \leq r$.

Now we prove (\ref{eqn:Prop:qlin}) when $(R,+,0)$ is not assumed torsion-free.
Choose a surjection of commutative rings $\pi: \bar{R} \twoheadrightarrow R$ such that the abelian group $(\bar{R},+,0)$ of $\bar{R}$ is torsion free, for instance, $\Z[R] \twoheadrightarrow R:\langle a\rangle \mapsto a$.
Choose a sequence $\bar{a}=(\bar{a}_1,...,\bar{a}_m)$ in $\bar{R}$ and a polynomial $\bar{p}(X)\in \bar{R}[X]$ such that $\pi(\bar{a}) = a$ and $\pi(\bar{p}(X)) = p(X)$.
The ring homomorphism $\pi$ makes $A$ and $B$ into $\bar{R}$-modules, and the action of $(\bar{R},\cdot,1)$ on $H_s(G,K)$ is induced from the $({R},\cdot,1)$-action via the map $\pi$.
Therefore, multiplication by the element $\bar{\sigma} = s_{\bar{p}}(\bar{a}) - \bar{p}(0)$ on $H_s(G,K)$ equals multiplication by the element $\sigma=s_{{p}}({a}) - {p}(0)$.
In particular, $\bar{\sigma}^{-1}H_s(G,K)={\sigma}^{-1}H_s(G,K)$.
By the torsion-free case above, we have
$\bar{\sigma}^{-1}H_s(G,K) = 0$ for $sd\max(\alpha,\beta)<m$. 
This finishes the proof of (\ref{eqn:Prop:qlin}) and hence that of the proposition.
\end{proof}

\begin{example}
\label{ex:MqSemiDirect}
Let $G$ be a group that acts from the left on an $R$-module $M$ through $R$-module homomorphisms.
Then for all $q\geq 0$, the semi-direct product 
$M(q)\rtimes G$ carries an action of $(R,\cdot,1)$ defined by 
$a(x,g)=(a^qx,g)$ such that the exact sequence 
$$1 \to M(q) \to M(q)\rtimes G \to G \to 1$$
is $(R,\cdot,1)$-equivariant with trivial action on the base $G$ and equivariant section $G \to  M(q)\rtimes G: g\mapsto (0,g)$.
By Proposition \ref{prop:HpLambdaM3} with $B=0$, $\alpha =q$, $A=M$, $N=M(q)$, the relative homology groups $H_s(M(q)\rtimes G,G)$ are quasi-linear $\Z_0[R]$-modules whenever $q\geq 1$.
\end{example}

\begin{example}
Continuing example \ref{ex:MqSemiDirect}, assume moreover that there is an integer $q\geq 1$ and a group homomorphism $\rho:R^* \to Z(G)$ into the center $Z(G)$ of $G$ such that $\rho(a)x=a^qx$.
Then the $(R,\cdot,1)$ action of $a\in R^*$ on $M(q)\rtimes G$ equals the conjugation action on $M(q)\rtimes G$ by $(0,\rho(a))$. 
In particular, the quasi-linear $\Z_0[R]$-module $H_s(M(q)\rtimes G,G)$ yields the trivial action when restricted to $R^*\subset \Z_0[R]$.
By Remark \ref{rmk:qlinTrivAction}, if $R$ is local with infinite residue field, we must have $H_s(M(q)\rtimes G,G)=0$.
This has been used many times, for instance for $G=GL_n(R)$ acting diagonally on $M=R^n \times \cdots \times R^n$ via its natural action on $R^n$ and $\rho:R^* \to GL_n(R):a \mapsto a\cdot I_n$, we obtain \cite[Theorem 1.11]{SuslinNesterenko} for local rings with infinite residue fields.
\end{example}

\begin{example}
Continuing example \ref{ex:MqSemiDirect}, we have $s_{\langle X^r\rangle}(a)^{-1}H_s(M(q)\rtimes G,G)=0$ for all sequences $a=(a_1,...,a_m)$ in $R$ with $m\geq m_0$.
This was used in \cite{myEuler} for $G=SL_n(R)$, $M=R^n \times \cdots \times R^n$, and $r$ and $q$ powers of $n$.
\end{example}

Now comes the most relevant example for this paper.

\begin{example}
\label{ex:SpQlinear}
Let $n\geq 0$ be an integer.
For $a\in R^*$, the conjugation action $c_a$ of the $(2n+2)\times (2n+2)$ diagonal matrix $D_a\in \Sp_{2n+2}(R)$ with diagonal entries $(a,a^{-1},1,1,...,1)$ on the group $\Sp_{2n+2}(R)$
induces an action
$$\left(\begin{smallmatrix}1& c & {^tu}\psi M  \\ 0 & 1 & 0 \\  0 & u & M \end{smallmatrix}\right)
\stackrel{c_a}{\mapsto}
\left(\begin{smallmatrix}a& 0 & 0  \\ 0 & a^{-1} & 0 \\  0 & 0 & 1 \end{smallmatrix}\right)\left(\begin{smallmatrix}1& c & {^tu}\psi M  \\ 0 & 1 & 0 \\  0 & u & M \end{smallmatrix}\right)
\left(\begin{smallmatrix}a^{-1}& 0 & 0  \\ 0 & a & 0 \\  0 & 0 & 1 \end{smallmatrix}\right)
=
\left(\begin{smallmatrix}1& a^2c & {^t(au)}\psi M  \\ 0 & 1 & 0 \\  0 & au & M \end{smallmatrix}\right)
$$
on the subgroup $\Sp_{2n+1}(R)$ which extends to an action 
$$\left(\begin{smallmatrix}1& c & {^tu}\psi M  \\ 0 & 1 & 0 \\  0 & u & M \end{smallmatrix}\right)
\stackrel{\langle a \rangle}{\mapsto} 
\left(\begin{smallmatrix}1& a^2c & {^t(au)}\psi M  \\ 0 & 1 & 0 \\  0 & au & M \end{smallmatrix}\right)$$
of the monoid $(R,\cdot,1)$ on $\Sp_{2n+1}(R)$, $a\in R$.
Denote by $N\subset \Sp_{2n+1}(R)$ the subgroup of matrices with $M=1$, by $A\subset \Sp_{2n+1}(R)$ the subgroup of matrices with $M=1$ and $c=0$, and by $B\subset \Sp_{2n+1}(R)$ the subgroup with $M=1$, $u=0$, then $(A,\cdot,1)=(R^{2n},+,0)$, $(B,\cdot,1)=(R,+,0)$, and we have $(R,\cdot,1)$ equivariant exact sequences
$$1 \to B(2) \to N \to A(1) \to 1, \hspace{6ex} 1 \to N \to \Sp_{2n+1}(R) \to \Sp_{2n}(R)\to 1$$
with left sequence central and $\langle 0 \rangle :\Sp_{2n+1}(R) \to \Sp_{2n+1}(R)$ the projection $\rho:\Sp_{2n+1}(R) \to \Sp_{2n}(R)$ followed by the inclusion $\eps:\Sp_{2n}(R) \to \Sp_{2n+1}(R)$.
By Proposition \ref{prop:HpLambdaM3}, the image of the projector $1-(\eps\rho)_*$ of $H_p(\Sp_{2n+1}(R))$ which is the relative homology group
$$\tilde{H}_p(\Sp_{2n+1}(R)):=H_p(\Sp_{2n+1}(R),\Sp_{2n}(R))=\im(1-(\eps\rho)_*) $$
is a quasi-linear $\Z_0[R]$-module for all $p\in \Z$.
We have a canonical decomposition
$$H_p(\Sp_{2n+1}(R)) = \im((\eps\rho)_*) \oplus \im(1-(\eps\rho)_*) = H_p(\Sp_{2n}(R)) \oplus \tilde{H}_p(\Sp_{2n+1}(R)).$$
\end{example}

\begin{lemma}
\label{lem:tildtoIszero}
Let $R$ be a local ring with infinite residue field.
Then the $\Z_0[R]$-module $\tilde{H}_p(\Sp_{2n+1}(R))=H_p(\Sp_{2n+1}(R),\Sp_{2n}(R))$
is quasi-linear, and the composition
$$\tilde{H}_p(\Sp_{2n+1}(R)) \subset H_p(\Sp_{2n+1}(R)) \to H_p(\Sp_{2n+2}(R))$$
is zero.
Moreover, the map $H_p(\Sp_{2n+1}(R)) \to H_p(\Sp_{2n+2}(R))$ is surjective if and only if the map 
$H_p(\Sp_{2n}(R)) \to H_p(\Sp_{2n+2}(R))$ is surjective.
\end{lemma}

\begin{proof}
Quasi-linearity is Example \ref{ex:SpQlinear}.

Note that the composition $\tilde{H}_p(\Sp_{2n+1}(R)) \to H_p(\Sp_{2n+2}(R))$
is $R^*$-equivariant where $R^*$ acts through conjugation with $D_a\in \Sp_{2n+2}$ and thus acts trivially on the target.
Since the source is quasi-linear, there is an integer $m_0\geq 1$ such that for all sequences $a=(a_1,...,a_m)$ of $m\geq m_0$ elements in $R$ we have $s_{X}^{-1}(a)\tilde{H}_p(\Sp_{2n+1}R) =0$.
If $R$ is local with infinite residue field, we can find a sequence $a=(a_1,...,a_m)$ such that $a_J\in R^*$ for all $\emptyset \neq J \subset \{1,...,m\}$.
Since $R^*$ acts on $H_p(\Sp_{2n+2}(R))$ trivially, for such an $a$, $s_{X}(a)$ acts as the identity on
$H_p(\Sp_{2n+2}(R))$ and thus $s_{X}^{-1}(a)H_p(\Sp_{2n+2}R) = H_p(\Sp_{2n+2}(R))$.
In particular, the $R^*$-equvariant map 
$\tilde{H}_p(\Sp_{2n+1}(R)) \to H_p(\Sp_{2n+2}(R))$
factors through $s_{X}^{-1}(a)\tilde{H}_p(\Sp_{2n+1}R) =0$, hence that map is zero.
For the last statement we note that $H_p(\Sp_{2n}(R)) \to H_p(\Sp_{2n+2}(R))$ is the localisation of 
$H_p(\Sp_{2n+1}(R)) \to H_p(\Sp_{2n+2}(R))$ at $s_{X}(a)$. 
In particular, surjectivity of the second map implies surjectivity of the first. 
The converse is obvious.
\end{proof}

\begin{corollary}
\label{cor:DirSumDec}
Let $R$ be a local ring with infinite residue field.
Under the decomposition $H_p(\Sp_{2n+1}) = H_p(\Sp_{2n}) \oplus \tilde{H}_p(\Sp_{2n+1})$ of Example \ref{ex:SpQlinear}, the maps 
$H_p(\Sp_{2n}(R)) \to H_p(\Sp_{2n+1}(R)) \to H_p(\Sp_{2n+2}(R))$
become
$$\xymatrix{
H_p(\Sp_{2n}(R)) \ar[r]^{\hspace{-10ex}\left(\begin{smallmatrix}1 \\ 0 \end{smallmatrix}\right)} &
H_p(\Sp_{2n}(R)) \oplus \tilde{H}_p(\Sp_{2n+1}(R)) \ar[r]^{\hspace{8ex}(\eps_*, 0)} & H_p(\Sp_{2n+2}(R)).
}$$
\end{corollary}

\begin{proof}
This follows from Lemma \ref{lem:tildtoIszero}.
\end{proof}

\section{Degeneration at $E^2$}
\label{sec:E2degeneration}

In this section we will prove that the spectral sequence (\ref{eqn:E1spseq}) degenerates at $E^2$.
Our strategy for degeneration is to construct a map of spectral sequences $\tilde{E} \to E$ from a spectral sequence $\tilde{E}$ to (\ref{eqn:E1spseq}). 
The spectral sequence $\tilde{E}$ will trivially degenerate at $E^2$, and the main point will be to show that $\tilde{E}^2 \to E^2$ is surjective in all bidegrees. 
That will ensure that (\ref{eqn:E1spseq}) degenerates at $E^2$ as well.
The spectral sequence $\tilde{E}$ will be a direct sum of spectral sequences $E(r)$, $r=0,...,n$, which we will introduce now.
\vspace{1ex}

%Recall that for a sequence $w=(w_1,...,w_q)$ of length $q$ of vectors in $R^{2n}$ we denote by $d_iw$ the sequence $d_iw = (w_1,..., \hat{w}_i,...,w_q)$ of length $q-1$ omitting the $i$-th vector, $1\leq i \leq q$.
For $0\leq r < n$ and $i=1,...,2r+2$, consider the $\Sp_{2n-2r}(R)$-set
$$U_{2r+2}^{(i)}(R^{2n}) =
\{ \left(\begin{smallmatrix}u \\ w \end{smallmatrix}\right)\in M_{2n,2r+2}(R)|\  u \in U_{2r+2}(R^{2r}), \ w_i \in U_1(R^{2n-2r}),\ d_iw=0\}
$$
where $N\in \Sp_{2n-2r}(R)$ acts by $N\cdot \left(\begin{smallmatrix}u \\ w \end{smallmatrix}\right) = \left(\begin{smallmatrix}u \\ Nw \end{smallmatrix}\right)$, that is, via its natural inclusion $\Sp_{2n-2r}(R) \subset \Sp_{2n}(R)$.
Note that $d_iw=(w_1,...,\hat{w}_i,...,w_{2r+2})=0$ means that $w=(0,..0,w_i,0..,0)$ only has potentially non-zero entry in the $i$-th column.
We have the bijection
$$\Sp_{2n-2r}(R)\backslash U_{2r+2}^{(i)}(R^{2n}) \stackrel{\cong}{\longrightarrow}  U_{2r+2}(R^{2r}): \left(\begin{smallmatrix}u \\ w \end{smallmatrix}\right) \mapsto u.$$
The stabiliser of the action on $U^{(i)}_{2r+2}(R)$ at $\left(\begin{smallmatrix}u \\ (e_{1})_i \end{smallmatrix}\right)$ is $\Sp_{2n-2r-1}(R)$
where $(e_1)_i=(0,..,0,e_1,0...,0)$ with $e_1\in R^{2n-2r}$ in the $i$-th column.
Note that if $v = \left(\begin{smallmatrix}u \\ w \end{smallmatrix}\right)\in U_{2r+2}^{(i)}$ then $d_jv\in U_{2r+1}(R^{2n})$ for all $1\leq j\leq 2r+2$ with $j\neq i$, and
$d_iv \in U_{2r+1}(R^{2r})$.
We define the complex ${C}_*(R^{2n};r)$ as
$$\xymatrix{
0 \ar[r] & \displaystyle \bigoplus_{i=1}^{2r+2} \Z[U^{(i)}_{2r+2}(R^{2n})] \ar[rrrr]^{\hspace{6ex}(-d_1,d_2,-d_3,...,d_{2r+2})} &&&& \Z [U_{2r+1}(R^{2r})] \ar[r] & 0 
}$$
with $\Z [U_{2r+1}(R^{2r})]$ placed in degree $2r$, and the $i$-th component of the differential is $(-1)^id_i$.
This is a complex of $\Sp_{2n-2r}(R)$-modules where $\Sp_{2n-2r}(R)$ acts trivially on the degree $2r$ piece $\Z [U_{2r+1}(R^{2r})] $.
Since $d\circ d=0$, the diagram
$$\xymatrix{
 0 \ar[r]\ar[d] & \Z[U^{(i)}_{2r+2}(R^{2n})] \ar[d]^{d^{\omit}_i} \ar[rr]^{\hspace{0ex}(-1)^id_i} && \Z  [U_{2r+1}(R^{2r})] \ar[d]^{d} \ar[r] & 0\ar[d] \\
 \Z[U_{2r+2}(R^{2n})] \ar[r]_d & \Z[U_{2r+1}(R^{2n})]  \ar[rr]_d &&  \Z[U_{2r}(R^{2n})] \ar[r]_d  & \Z[U_{2r-1}(R^{2n})]}$$
commutes 
where the second to left vertical map $d^{\omit}_{i}: \Z[U^{(i)}_{2r+2}(R^{2n})] \to  \Z[U_{2r+2}(R^{2n})]$ is defined on basis elements $w\in  \Z[U^{(i)}_{2r+2}(R^{2n})]$ by
$$d_i^{\omit}(w) = \sum_{j=1, j\neq i}^{2r+2}(-1)^{j+1}d_jw$$
and can informally be thought of as $d_i^{\omit} = d + (-1)^id_i$.
This defines the map 
of complexes $\ffi: C_*(R^{2n};r) \to C_{*}(R^{2n})$ of $\Sp_{2n-2r}(R)$-modules (where we have suppressed some of the entries $R^{2n}$)
$$\xymatrix{
 0 \ar[r]\ar[d] & \displaystyle \bigoplus_{i=1}^{2r+2}\Z[U^{(i)}_{2r+2}] \ar[d]^{(d^{\omit}_i)_i} \ar[rr]^{\hspace{1ex}((-1)^id_i)_i} && \Z  [U_{2r+1}(R^{2r})] \ar[d]^{d} \ar[r] & 0\ar[d] \\
\cdots\to \Z[U_{2r+2}] \ar[r]_d & \Z[U_{2r+1}]  \ar[rr]_d &&  \Z[U_{2r}(R^{2n})] \ar[r]_d  & \Z[U_{2r-1}]\to\cdots}$$
For $r=n$, we let $C_*(R^{2n},n)$ be the complex $\Z[U_{2n+1}(R^{2n})][2n]$ concentrated in degree $2n$ and define the map of complexes $\ffi:C_*(R^{2n},n) \to C_*(R^{2n})$ in degree $n$ as the map $d:\Z[U_{2n+1}(R^{2n})]\to \Z[U_{2n}(R^{2n})]$.
For $0 \leq r \leq n$, the pair 
$$(\eps,\ffi):(\Sp_{2n-2r}(R), C_*(R^{2n};r)) \longrightarrow (\Sp_{2n}(R), C_{*}(R^{2n}))$$
defines a map of associated group homology spectral sequences 
\begin{equation}
\label{eqn:SpSeqMapp}
E_{p,q}^s(R^{2n};r) \longrightarrow E_{p,q}^s(R^{2n})
\end{equation}
resulting from the filtrations by degree $C_{\leq q}(R^{2n};r)$ and $C_{\leq q}(R^{2n})$ of the coefficient complexes $C_*(R^{2n};r)$ and $C_{*}(R^{2n})$.
By definition, we have
$$E_{p,q}^s(R^{2n};r)=0,\hspace{3ex} q\neq 2r,2r+1.$$
In particular, the spectral sequences $E(R^{2n};r)$ degenerate at the $E^2$-page.

The following result shows that the spectral sequence (\ref{eqn:E1spseq}) degenerates at $E^2$.

\begin{proposition}
\label{prop:Esurjective}
Let $R$ be a local ring with infinite residue field.
For all integers $0 \leq r \leq n$, $s=2$, $q=2r, 2r+1$ and all $p\in \Z$, the map (\ref{eqn:SpSeqMapp}) is surjective:
$$E_{p,q}^2(R^{2n};r) \twoheadrightarrow E_{p,q}^2(R^{2n}),\hspace{2ex}q=2r, 2r+1.$$
In particular, the spectral sequence (\ref{eqn:E1spseq}) degenerates at $E^2$.
\end{proposition}

\begin{proof}[Proof of Proposition \ref{prop:Esurjective} for $q=2r$]
The map $E^1_{p,2r}(R^{2n};r) \to  E^1_{p,2r}(R^{2n})$ is the first map in the complex
$$\xymatrix{
H_p(\Sp_{2n-2r})  \otimes \Z[U_{2r+1}(R^{2r})]  \ar[r]^{1\otimes d\circ \Gamma} & H_p(\Sp_{2n-2r}) \otimes  \Z[\Skew_{2r}^+] \ar[d]^{\eps_*\otimes d}\\
&   H_p(\Sp_{2n-2r+1})\otimes \Z[\Skew_{2r-1}^+];
}$$
see Corollary \ref{cor:E1IdentWellDef}.
In view of Lemma \ref{lem:commdiag}, the second map in that complex is $d^1_{p,2r}: E^1_{p,2r}(R^{2n}) \to E^1_{p,2r-1}(R^{2n})$.
Since $\eps_*:H_p(\Sp_{2n-2r}) \to H_p(\Sp_{2n-2r+1})$ is (split) injective,
Lemmas \ref{lem:GammaBij} and \ref{lem:SkewExact} imply that this complex is exact.
It follows that $E^1_{p,2r}(R^{2n};r)$ surjects onto the kernel of the right vertical map which which surjects onto $E^2_{p,2r}(R^{2n})$.
In particular, its quotient $E^2_{p,2r}(R^{2n};r)$ surjects onto $E^2_{p,2r}(R^{2n})$.
\end{proof}

The case $q=2r+1$ of Proposition \ref{prop:Esurjective} is somewhat more involved except when $r=n$ in which case the map $0=E^1_{p,2n+1}(R^{2n};n) \to E^1_{p,2n+1}(R^{2n})=0$ is clearly surjective.
So assume $0\leq r<n$.
For $i=1,..., 2r+2$, consider the map 
%$$
\begin{equation}
\label{eqn:map4.1}
\xymatrix{
\gamma_i:H_p(\Sp_{2n-2r-1}) \otimes \Z[U_{2r+2}(R^{2r})]  
\ar[r]
&
H_p(\Sp_{2n-2r-1}) \otimes \Z[\Skew^+_{2r+1}] 
}
\end{equation}
%$$
which for $u\in U_{2r+2}(R^{2r})$ and $\alpha \in H_p(\Sp_{2n-2r-1})$ is defined by
$$\gamma_i(\alpha \otimes u) = \sum_{1\leq j \neq i \leq 2r+2}(-1)^{j+1} \left(c^{-1}_{\delta_{ij}\det (u^{\omit}_{ij})}\right)_*(\alpha)\otimes d_j\Gamma(u)$$
where
$u^{\omit}_{ij}$ is obtained from $u$ by omitting the $i$-th and $j$-th columns, 
$c_a$ is conjugation with the diagonal matrix $(a,a^{-1},1,...,1) \in \Sp_{2n-2r}(R)$ for $a\in R^*$,
and $\delta_{ij}$ is defined by
$$\delta_{ij}= \left\{\begin{array}{clc} (-1)^{i+1}&, & i<j\\ 0&, & i=j \\ (-1)^{i}&, & i>j\end{array}\right.:
\hspace{6ex}
(\delta_{ij})=\left(\begin{smallmatrix} 
0 & + & + & + & + & \cdots & +\\
+ & 0 & - & -  &  - & \cdots &  -\\
-  & - & 0 & + & + & \cdots & +\\
+ & + & + & 0  &  - & \cdots &  -\\
   &    &    &     &  & \ddots &  
\end{smallmatrix}
\right).
$$
%\delta_{ij}=(-1)^{i+1}$ if $i<j$ and $\delta_{ij}=(-1)^{i}$ if $i>j$.

\begin{lemma}
\label{lem:E1DiagIsomTo}
The commutative diagram
$$\xymatrix{
E_{p,2r+1}^1(R^{2n};r) \ar[r]^{d^1} \ar[d]_{(\text{\ref{eqn:SpSeqMapp}})} & E_{p,2r}^1(R^{2n};r)\ar[d]^{(\text{\ref{eqn:SpSeqMapp}})} \\
E_{p,2r+1}^1(R^{2n})  \ar[r]_{d^1} & E_{p,2r}^1(R^{2n})}
$$
is isomorphic to the commutative diagram
$$\xymatrix{
\displaystyle \bigoplus_{i=1}^{2r+2} H_p(\Sp_{2n-2r-1}) \otimes \Z[U_{2r+2}(R^{2r})]  
\ar[rrr]^{((-1)^i\eps_*\otimes d_{i})_i}
\ar[d]_{\gamma=(\gamma_1,\gamma_2,...,\gamma_{2r+2})}&&&
H_p(\Sp_{2n-2r}) \otimes \Z[U_{2r+1}(R^{2r})]
\ar[d]^{1\otimes \Gamma(d) }\\
H_p(\Sp_{2n-2r-1}) \otimes \Z[\Skew^+_{2r+1}]  
\ar[rrr]_{\eps_*\otimes d}  
& && 
H_p(\Sp_{2n-2r}) \otimes \Z[\Skew^+_{2r}].}
$$
\end{lemma}

\begin{proof}
The right vertical and the lower horizontal map have already been identified in Lemmas \ref{cor:E1IdentWellDef} and \ref{lem:commdiag}.
For the other two maps, we note that 
$$E_{p,2r+1}^1(R^{2n};r) = \bigoplus_{i=1}^{2r+2} H_p(\Sp_{2n-2r},\Z[U^{(i)}_{2r+2}(R^{2n})]).$$
By Shapiro's Lemma, we obtain the isomorphism
$$ \sum_u \left(\eps, \left(\begin{smallmatrix}u \\ (e_1)_i \end{smallmatrix}\right)\right)_*:\bigoplus_{u \in U_{2r+2}(R^{2r})} H_p(\Sp_{2n-2r-1}) \stackrel{\cong}{\longrightarrow}   H_p(\Sp_{2n-2r},\Z[U^{(i)}_{2r+2}(R^{2n})])$$
This yields the identification of the top horizontal map.
Composing with the map $E^1_{p,2r+1}(R^{2n};r) \to E^1_{p,2r+1}(R^{2n})$ yields the map
\begin{equation}
\label{eqn:SumUSpSeqMap2}
 \bigoplus_{u \in U_{2r+2}(R^{2r})} H_p(\Sp_{2n-2r-1}) \longrightarrow   H_p(\Sp_{2n},\Z[{U}_{2r+1}(R^{2n})])
\end{equation}
which is
$$
\sum_{1\leq j \neq i \leq 2r+2}(-1)^{j+1} \left(\eps, d_j\left(\begin{smallmatrix}u \\ (e_1)_i \end{smallmatrix}\right)\right)_*
$$
on the component corresponding to $u\in U_{2r+2}(R^{2r})$.
We recall the isomorphism
\begin{equation}
\label{eqn:E1iso}
 \bigoplus_{A \in \Skew_{2r+1}^+} H_p(\Sp_{2n-2r-1}) \stackrel{\cong}{\longrightarrow} H_p(\Sp_{2n},\Z[{U}_{2r+1}(R^{2n})])
 \end{equation}
from Lemma \ref{cor:E1IdentWellDef} which is 
$(\eps\circ c_{\det v},v)_*$ on the component corresponding to $A \in \Skew_{2r+1}^+(R)$
where
$v\in U_{2r+1}(R^{2n})$ satisfies $\Gamma(v)=A$ and generates $R^{2r+1}$.
For $u\in U_{2r+2}(R^{2r})$ and $j\neq i$, the unimodular sequence $w=d_j \left(\begin{smallmatrix} u \\ (e_1)_i\end{smallmatrix}\right)$ generates $R^{2r+1}$.
Since 
$$\det w = \det \left( d_j \left(\begin{smallmatrix} u \\ (e_1)_i\end{smallmatrix}\right) \right) =
 \delta_{ij} \det u^{\omit}_{ij},$$
the diagram
$$\xymatrix{
H_p(\Sp_{2n-2r-1}) \ar[d]_{(c^{-1}_{\delta_{ij}\det u^{\omit}_{ij}})_*} \ar[rr]^{(\eps, \left(\begin{smallmatrix} u \\ (e_1)_i\end{smallmatrix}\right))_*} &&
H_p(\Sp_{2n-2r},\Z[U^{(i)}_{2r+2}]) \ar[d]^{d_j} \\
H_p(\Sp_{2n-2r-1}) \ar[rr]_{(\eps \circ c_{\det w},w)_*} &&
H_p(\Sp_{2n},\Z[U_{2r+1}])
}$$
commutes.
Since 
$$\Gamma(w) = \Gamma\left( d_j \left(\begin{smallmatrix} u \\ (e_1)_i\end{smallmatrix}\right) \right) = \Gamma( d_ju) $$
we apply Lemma \ref{lem:vIndep} to identifies the left vertical map in the lemma with $\gamma$.
\end{proof}

\begin{proof}[Proof of Proposition \ref{prop:Esurjective} for $q=2r+1$]
We need to show that the map of horizontal complexes 
$$\xymatrix{
0 \ar[r] \ar[d] & {E}_{p,2r+1}^1(R^{2n};r) \ar[r]^{d^1} \ar[d]_{(\text{\ref{eqn:SpSeqMapp}})} &
{E}_{p,2r}^1(R^{2n};r) \ar[d]^{(\text{\ref{eqn:SpSeqMapp}})} \\
E_{p,2r+2}^1(R^{2n})\ar[r]_{d^1}  & E_{p,2r+1}^1(R^{2n})\ar[r]_{d^1}  & E_{p,2r}^1(R^{2n})}$$
is surjective on homology (at the middle term).
By Corollary \ref{cor:DirSumDec} and Lemma \ref{lem:E1DiagIsomTo}, this map of complexes is isomorphic to the direct sum of 
\begin{equation}
\label{eqn:Htilde}
\xymatrix{
0 \ar[r] \ar[d] & \displaystyle \bigoplus_{i=1}^{2r+2} \widetilde{H}_p(\Sp_{2n-2r-1}) \otimes \Z[U_{2r+2}(R^{2r})] \ar[d]_{\gamma} \ar[r] & 0 \ar[d]&\\
0 \ar[r] & 
\widetilde{H}_p(\Sp_{2n-2r-1}) \otimes \Z[\Skew^+_{2r+1}] 
\ar[r] & 0 & \text{and}
}\end{equation}
\begin{equation}
\label{eqn:ABdiagram}
\xymatrix{
0 \ar[r] \ar[d] & \displaystyle \bigoplus_{i=1}^{2r+2} A \otimes \Z[U_{2r+2}(R^{2r})] \ar[d]^{\hspace{0ex}1\otimes (\Gamma(d)+(-1)^i\Gamma(d_i))_i } \ar[rr]^{(-1)^i\eps_*\otimes d_i} && B \otimes \Z[U_{2r+1}(R^{2r})] \ar[d]^{1\otimes \Gamma(d)}\\
A \otimes \Z[\Skew^+_{2r+2}]  \ar[r]_{1\otimes d} & 
A \otimes \Z[\Skew^+_{2r+1}] 
\ar[rr]_{\eps_*\otimes d} && B\otimes \Z[\Skew^+_{2r}]}
\end{equation}
where $A=H_p(\Sp_{2n-2r-2})$ and $B=H_p(\Sp_{2n-2r})$.
For the latter, we use that $c_a$ is the identity on $A$.
Proposition \ref{prop:Esurjective} now follows from Lemmas \ref{lem:2ndDiagSurj} and \ref{lem:Htilde} below.
\end{proof}

\begin{lemma}
\label{lem:2ndDiagSurj}
The map of complexes (\ref{eqn:ABdiagram}) is surjective in homology.
\end{lemma}

\begin{proof}
Let $F$ be the image of the map $\Gamma(d):\Z[U_{2r+1}(R^{2r})] \to \Z[\Skew_{2r}^+]$.
This is a free $\Z$-module, and it is also the image of $d:\Z[\Skew^+_{2r+1}] \to \Z[\Skew_{2r}^+]$.
In the diagram (\ref{eqn:ABdiagram}), we can replace $\Z[\Skew_{2r}^+]$ with $F$ and 
the lower left horizontal arrow $1\otimes d$ with its cokernel $0 \to \coker (1\otimes d)$ without changing homology since that cokernel is  $A \otimes F$, by Lemma \ref{lem:SkewExact}.
Thus, we can replace the diagram (\ref{eqn:ABdiagram}) with the diagram
\begin{equation}
\label{eqn:EKerSurj1}
\xymatrix{
0 \ar[r] \ar[d] & \displaystyle \bigoplus_{i=1}^{2r+2} A \otimes \Z[U_{2r+2}(R^{2r})] \ar@{->>}[d]^{\hspace{0ex}1\otimes (\Gamma(d)+(-1)^i\Gamma(d_i))_i } \ar[rr]^{(-1)^i\eps_*\otimes d_i} && B \otimes \Z[U_{2r+1}(R^{2r})] \ar@{->>}[d]^{1\otimes \Gamma(d)}\\
0  \ar[r]& 
A \otimes F
\ar[rr]_{\eps_*\otimes 1} && B\otimes F}
\end{equation}
without changing homology.
The right hand square is obtained by tensoring the diagram of free abelian groups
\begin{equation}
\label{eqn:EKerSurj2}
\xymatrix{
 \displaystyle \bigoplus_{i=1}^{2r+2}  \Z[U_{2r+2}(R^{2r})] 
 \ar@{->>}[d]_{(\Gamma(d)+(-1)^i\Gamma(d_i))_i } \ar@{->>}[rr]^{(-1)^i d_i} && \Z[U_{2r+1}(R^{2r})] \ar@{->>}[d]^{ \Gamma(d)}\\
  F
\ar[rr]_{1} &&  F}
\end{equation}
with the map $\eps_*:A \to B$.
The top horizontal arrow in (\ref{eqn:EKerSurj2}) is surjective
because the maps $d_i:U_{2r+2}(R^{2r}) \to U_{2r+1}(R^{2r})$ are surjective.
Since all abelian groups in diagram (\ref{eqn:EKerSurj2}) are free, that diagram is isomorphic to
$$\xymatrix{
M\oplus N \ar[r]^{(1,0)} \ar@{->>}[d]_{(f,0)} & M \ar@{->>}[d]^{f}\\
F \ar[r]_1 & F}$$
where $M=  \Z[U_{2r+1}(R^{2r})]$ and $N$ is the kernel of the top horizontal arrow.
It follows that diagram (\ref{eqn:EKerSurj1}) is isomorphic to
$$\xymatrix{
0 \ar[d] \ar[r] & A \otimes (M\oplus N) \ar[rr]^{\eps_*\otimes (1,0)} \ar@{->>}[d]_{1\otimes (f,0)} && B \otimes M \ar@{->>}[d]^{1\otimes f}\\
0 \ar[r] & A \otimes F \ar[rr]_{\eps_*\otimes 1} && B \otimes F.}$$
Hence the map on homology (kernels of right horizontal maps) is
$$ (1 \otimes f, 0):  \left(\ker(\eps_*) \otimes M\right) \oplus (A \otimes N) \longrightarrow \ker(\eps_*) \otimes F$$
which is surjective since $f$ is.
\end{proof}

For a $\Z_0[R]$-module $H$, define the map of $\Z_0[R]$-modules, generalising (\ref{eqn:map4.1}),
\begin{equation}
\label{eqn:gammaH}
\gamma = (\gamma_1,\gamma_2...,\gamma_{2r+2}):\bigoplus_{i=1}^{2r+2}H \otimes_{\Z} \Z[U_{2r+2}(R^{2r})]  \to H\otimes_{\Z} \Z[\Skew^+_{2r+1}(R)] 
\end{equation}
by
\begin{equation}
\label{eqn:Skew0deltaHatMap3}
\gamma_i(h\otimes u) = \sum_{1\leq j \neq i \leq 2r+2}(-1)^{j+1}\langle \delta_{ij}{\det}^{-1}  u^{\omit}_{ij}\rangle \cdot h\otimes \Gamma(d_ju)
\end{equation}
for $u\in U_{2r+2}(R^{2r})$ and $h\in H$.
Recall from Lemma \ref{lem:tildtoIszero} that the relative homology groups $\widetilde{H}_p(\Sp_{2n+1}(R))=H_p(\Sp_{2n+1}(R),\Sp_{2n}(R))$ are quasi-linear $\Z_0[R]$-modules.

\begin{lemma}
\label{lem:Htilde}
Let $R$ be a local ring with infinite residue field, and let $r\geq 0$ be an integer.
Then for all quasi-linear $\Z_0[R]$-modules $H$, the map (\ref{eqn:gammaH}) is surjective.
In particular, the map of complexes (\ref{eqn:Htilde}) is surjective in homology.
\end{lemma}

\begin{proof}
We may write $h[B]$ in place of $h\otimes B$.
Denote by $N=\coker(\gamma)$ the cokernel of $\gamma$.
We have to show that $N=0$.
As a cokernel of a $\Z_0[R]$-linear map of quasi-linear modules, $N$ is also quasi-linear.
In $N$, the expressions on the right hand side of (\ref{eqn:Skew0deltaHatMap3}) are zero.
In matrix form, the system of equations, expressing the right hand side of (\ref{eqn:Skew0deltaHatMap3}) as zero, can be written as $M(U)\cdot X(U)=0$ for $U\in U_{2r+2}(R^{2r})$ and $h\in H$ where
$$M(U)=\begin{pmatrix}\langle\delta_{ij} {\det}^{-1}U^{\wedge}_{i,j}\rangle \end{pmatrix}$$ is the $(2r+2)\times (2r+2)$ matrix with entries in $\Z_0[R]$ which has $0$'s on the diagonal and 
$\langle\delta_{ij}{\det}^{-1}U^{\wedge}_{i,j}\rangle$ at the $i,j$-spot, and $X(U)=( (-1)^{j+1}z\ [\Gamma(U^{\wedge}_{j})])$ is the column vector with $(-1)^{j+1}z\ [\Gamma(U^{\wedge}_{j})]$ at its $j$-th entry.
Multiplying with the adjugate of $M(U)$ yields the equation
$(\det M(U))\ h\ [\Gamma(U^{\wedge}_{j})] = 0 \in N$.
Thus, for $h\in H$ and $B\in \Skew_{2r+1}^+(R)$ we have 
\begin{equation}
\label{eqn:MuB0}
(\det M(U,x))\cdot h\cdot [B] =0  \in N
\end{equation}
 for all $U\in U_{2r+1}(R^{2r})$, $x\in R^{2r}$ such that $\Gamma(U)=B$ and $(U,x)\in U_{2r+2}(R^{2r})$.
The following Lemma \ref{lem:DetMgenerates} therefore shows that $h [B]= 0\in N$ for all $h\in H$ and $B\in \Skew_{2r+1}^+(R)$, that is, the map $\gamma$ in Lemma \ref{lem:Htilde} is surjective.
\end{proof}

\begin{lemma}
\label{lem:DetMgenerates}
For all $B\in \Skew_{2r+1}^+(R)$ and $h\in H$,
the radical of the annihilator ideal 
$$\sqrt{\ann(h [B])} \subset \Z_0[R]$$
of $h[B] \in N=\coker(\gamma)$ is the unit ideal.
\end{lemma}

\begin{proof}
Denote by $F$ the residue field of $R$, and by $\bar{x}\in F^s$ the reduction modulo the maximal ideal of the element $x\in R^s$.

For $B\in \Skew_{2r+1}^+(R)$, choose a normal form $U=(u_1,...,u_{2r+1}) \in U_{2r+1}(R^{2r})$ of $B$, that is, 
$\Gamma(U)=B$, $(u_1,...,u_{2r})$ is upper triangular, $(u_{2i-1})_{2i-1}=1$ and $(u_{2i})_{2i-1}=0$ for $i=1,...,r$; see Lemma \ref{lem:normalForm}.
For $\ell = 1,...,r$, the matrix $U(\ell)$ obtained from $U$ by deleting the first $2r - 2\ell$ rows and columns is in $U_{2\ell+1}(R^{2\ell})$.
Indeed, the sequence $U(\ell)$ is unimodular in $R^{2\ell}$ because $(u_1,...,u_{2r})$ is upper triangular and $(u_1,...,u_{2r},u_{2r+1})$ is unimodular.
It is non-generate as for $I\subset \{2r-2\ell +1,...,{2r+1}\}$ of even cardinality, the sequence $U(\ell)_{I-2r+2\ell}$ generates the orthogonal complement of $u_1,...,u_{2\ell}$ in the non-degenerate space generated by $(u_1,...,u_{2\ell},U_I)$ and is thus non-degenerate.

We will show by descending induction on $\ell = 1,...,r$ that 
\begin{equation}
\label{eqn:Surj:IndAss}
\det M(U(\ell),x) \in \sqrt{\ann(h [B])}
\end{equation}
for all $x\in R^{2\ell}$ such that $(U(\ell),x) \in U_{2\ell+2}(R^{2\ell})$.
%At the end we will show that the case $\ell=1$ generates the unit ideal.

The case $\ell = r$ is (\ref{eqn:MuB0}).
Let $\ell \in \{1,...,r-1\}$ and assume (\ref{eqn:Surj:IndAss}) holds for $\ell +1$ in place of $\ell$.
We want to show that (\ref{eqn:Surj:IndAss}) holds for $\ell$.
Fix $x\in R^{2\ell}$ such that $(U(\ell),x) \in U_{2\ell+2}(R^{2\ell})$.
For $\xi = (s,t,x)\in R \times R \times R^{2\ell}$, the matrix
$$(U(\ell+1),\xi) = 
\left(
\renewcommand\arraystretch{1.5}
\begin{array}{cc|ccc|c}
1 & 0 & \cdots & \ast & \cdots & s\\
0 & \alpha & \cdots & \ast & \cdots& t \\
\hline
0 & 0 & &&& x_1\\
\vdots & \vdots & & U(\ell) && \vdots \\
0 & 0 & &&& x_{2\ell}
\end{array}
\right)
$$
is in $U_{2\ell+4}(R^{2\ell +2})$ if and only if 
for all $1 \leq i<j \leq 2\ell+3$, the square matrix 
 $$(U(\ell+1)^{\wedge}_{ij},\xi)$$%=(u_1,...,\hat{u}_i,...,\hat{u}_j,...,u_{2r+1},\xi)$$
  is invertible, and for all $I \subset \{1,...,2\ell+3\}$ of odd cardinality $<2\ell+2$, the subspace spanned by $(U(\ell+1)_I,\xi)$ is non-degenerate.
This happens if and only if $\bar{s},\bar{t}\in F$ is not a solution to any of the equations in $F$
\begin{equation}
\label{eqn:DetPfxZero}
L_{ij}(s,t):=\det (U(\ell+1)^{\wedge}_{ij},\xi)=0\text{ and }\Pf(\Gamma(U(\ell+1)_I,\xi))=0
\end{equation}
where $1 \leq i<j \leq 2\ell+3$ and $I \subset \{1,...,2\ell+3\}$ of odd cardinality $<2\ell+2$.
Here, $\Pf(A)$ denotes the Pfaffian of a skew-symmetric matrix $A$.
The equations in (\ref{eqn:DetPfxZero}) are linear and homogeneous in $\xi$, hence, linear (possibly inhomogeneous) in $(s,t)\in R^2$.

We check that every equation in (\ref{eqn:DetPfxZero}) is non-trivial in $(s,t)$, that is, that for each equation in (\ref{eqn:DetPfxZero}), there is $(s,t)\in R^2$ for which the left-hand side of that equation does not vanish in $F$.
We start by investigating the Pfaffian equations. 
Let $I \subset \{1,...,2\ell+3\}$ be a subset of odd cardinality $<2\ell+2$.
By abuse of notation I will label the columns of $U(\ell)$ by $(U_3(\ell),...,U_{2r+3}(\ell))$ 
so that $U(\ell)_J$ is obtained from $U(\ell+1)_J$ by deleting the first two rows provided $J\subset \{3,4,...,2\ell+3\}$.
If $1,2 \in I$, then the subspace spanned by $(U(\ell+1)_I,\xi)$ is non-degenerate as it equals the subspace generated by 
$$\left(
\renewcommand\arraystretch{1.5}
\begin{array}{cc|ccc|c}
1 & 0 &  & 0 &  & 0\\
0 & \alpha &  & 0 & & 0 \\
\hline
%0 & 0 & &&& \\
0 & 0 & & U(\ell)_{I-\{1,2\}} && x %\\
%0 & 0 & &&& 
\end{array}
\right)
$$
which is non-degenerate since $(U(\ell),x)\in U_{2\ell+2}(R^{2\ell})$.
Hence, $\Pf(U(\ell+1)_I,\xi)$ is a unit in $R$ for all $s,t\in R$.
If $1\in I$ but $2 \notin I$ then the subspace spanned by $(U(\ell+1)_I,\xi)$ equals the subspace spanned by
$$\left(
\renewcommand\arraystretch{1.5}
\begin{array}{c|c|ccc}
1  & 0&  & 0 &  \\
\hline
0  & t&  & y &  \\
\hline
%0 & 0 & &&& \\
0  & x& & U(\ell)_{I-\{1\}} & %\\
%0 & 0 & &&& 
\end{array}
\right)
$$
which has Gram matrix
$$\left(
\renewcommand\arraystretch{1.5}
\begin{array}{c|c|c}
0  & t& y   \\
\hline
-t  & 0&\langle x,U(\ell)_{I-\{1\}}\rangle   \\
\hline
%0 & 0 & &&& \\
-^t\!y  & -^t\!\langle x,U(\ell)_{I-\{1\}}\rangle &  \Gamma(U(\ell)_{I-\{1\}})  %\\
%0 & 0 & &&& 
\end{array}
\right)
$$
with Pfaffian $t\Pf(\Gamma(U(\ell)_{I-\{1\}})) + c$ where $c$ does not depend on $t$.
Since $\Pf(\Gamma(U(\ell)_{I-\{1\}}))\neq 0 \in F$, for all $s\in R$ there is a $t\in R$ such that $\Pf(U(\ell+1)_I,\xi)$ is a unit in $R$.
%This defines a non-trivial hyperplane in $F^2$ since $\Pf(\Gamma(U(\ell)_{I-\{1\}}))\neq 0 \in F$.
If $2\in I$ but $1 \notin I$ then the subspace spanned by $(U(\ell+1)_I,\xi)$ equals the subspace spanned by
$$\left(
\renewcommand\arraystretch{1.5}
\begin{array}{c|c|ccc}
0  & s &  & y &  \\
\hline
\alpha  & 0 &  & 0 &  \\
\hline
%0 & 0 & &&& \\
0  & x& & U(\ell)_{I-\{2\}} & %\\
%0 & 0 & &&& 
\end{array}
\right)
$$
since $\alpha \in R^*$. 
This has Gram matrix
$$\left(
\renewcommand\arraystretch{1.5}
\begin{array}{c|c|c}
0  & -\alpha s& -\alpha y   \\
\hline
\alpha s  & 0&\langle x,U(\ell)_{I-\{2\}}\rangle   \\
\hline
%0 & 0 & &&& \\
\alpha ^t\!y  & -^t\!\langle x,U(\ell)_{I-\{2\}}\rangle &  \Gamma(U(\ell)_{I-\{2\}})  %\\
%0 & 0 & &&& 
\end{array}
\right)
$$
with Pfaffian $-\alpha s\Pf(\Gamma(U(\ell)_{I-\{2\}})) + c$ where $c$ does not depend on $s$.
Since $\alpha \Pf(\Gamma(U(\ell)_{I-\{2\}}))\neq 0 \in F$, for all $t$ there is $s$ such that 
$\Pf(U(\ell+1)_I,\xi)$ is a unit in $R$.
If $1,2\notin I$, assume first that $|I| \neq 2\ell +1$, hence $1 \leq |I|\leq 2\ell -1$.
Let $J \subset I$ be the subset obtained from $I$ by deleting its maximal element.
Then 
\begin{equation}
\label{eqn:SomeGramUell}
(U(\ell+1)_I,\xi)=
\left(
\renewcommand\arraystretch{1.5}
\begin{array}{ccc|c|c}
  & v &  &a & s\\
 & w & & b & t \\
\hline
 & U(\ell)_J && y & x 
\end{array}
\right)
\end{equation}
We need to find $s,t\in R$ such that the Pfaffian of (\ref{eqn:SomeGramUell}) is a unit in $R$, that is, such that the columns of (\ref{eqn:SomeGramUell}) span a non-degenerate subspace of $R^{2\ell +2}$.
Since $U(\ell)_J$ spans a non-degenerate subspace of $R^{2\ell}$, there are unique $a_j,b_j\in R$, $ j \in J$, such that 
$$\langle U_i(\ell),x\rangle = \sum_{j\in J}a_j\langle U_i(\ell),U_j(\ell)\rangle, \hspace{3ex}
\langle U_i(\ell),y\rangle = \sum_{j\in J}b_j\langle U_i(\ell),U_j(\ell)\rangle$$
for all $i\in J$.
Set 
$$x_0=\sum_{j\in J}a_jU_j(\ell),\hspace{2ex} s_0 = \sum_{j\in J}a_jv_j,
\hspace{2ex} t_0 = \sum_{j\in J}a_jw_j,$$
$$y_0=\sum_{j\in J}b_jU_j(\ell),\hspace{2ex} a_0 = \sum_{j\in J}b_jv_j,
\hspace{2ex} b_0 = \sum_{j\in J}b_jw_j.$$
Then the columns of (\ref{eqn:SomeGramUell}) and those of
\begin{equation}
\label{eqn:SomeGramUell2}
\left(
\renewcommand\arraystretch{1.5}
\begin{array}{ccc|c|c}
  & v &  &a-a_0 & s-s_0\\
 & w & & b-b_0 & t-t_0 \\
\hline
 & U(\ell)_J && y-y_0 & x-x_0
\end{array}
\right)
\end{equation}
span the same subspace of $R^{2\ell+2}$.
Moreover, $y-y_0, x-x_0$ is a basis of the orthogonal complement of $U(\ell)_J$ inside the non-degenerate subspace of $R^{2\ell}$ generated by $(U(\ell)_J,y,x) = (U(\ell)_I,x)$.
In particular, $c:=\langle y-y_0,x-x_0\rangle\in R^*$. 
For $(s,t)=(s_0,t_0)$, the Gram-matrix of (\ref{eqn:SomeGramUell2}) is
$$\left(
\renewcommand\arraystretch{1.5}
\begin{array}{c|c|c}
  \Gamma (U(\ell+1)_J)& \ast & 0\\
 \hline \ast & 0 & c \\ \hline 0 & -c & 0
 \end{array}
\right)
$$
which has Pfaffian $c\Pf(U(\ell +1)_J) \in R^*$.
In particular, the Pfaffian of (\ref{eqn:SomeGramUell}) is a unit for $(s,t)=(s_0,t_0)$.
If $|I|=2\ell+1$ (and $1,2\notin I$) then the space generated by $(U(\ell+1)_I,\xi)$ is non-degenerate if and only if the determinant $L_{12}(s,t)$ of $(U(\ell+1)_I,\xi)$ is a unit.
This is a special case of the linear equations $L_{ij}(s,t)$, $1 \leq i<j \leq 2\ell+3$, which we investigate now.
We have
$$L_{12}(s,t) = \det (U(\ell+1),\xi)^{\wedge}_{12} = as+bt +c$$
for some $c\in R$ where $a=-\det(A)$, $b=\det B$, 
$A$ is obtained from $U(\ell+1)^{\wedge}_{12}$ by deleting the first row, and
$B$ is obtained from $U(\ell+1)^{\wedge}_{12}$ by deleting the second row.
The matrices $A$ and $B$ are invertible because $U(\ell +1) \in U_{2\ell+3}(R^{2\ell+2})$, $U_1(\ell +1)=e_1$ and $U_2(\ell +1)=\alpha e_2$, $\alpha\in R^*$.
In particular, $a$ and $b$ are units, and there is $(s,t)\in R^2$ such that $L_{12}(s,t)\in R^*$.
For $i=1,2$ and $3 \leq j \leq 2\ell+3$, we have
$$L_{1j}(s,t)= a_{1j}s + c_{1j},\hspace{3ex}\text{and}\hspace{3ex}L_{2j}(s,t) = b_{2j}t +c_{2i},$$
where $a_{1j}=-\alpha\det U(\ell)^{\wedge}_{j}$ and $b_{2j}=\det U(\ell)^{\wedge}_{j}$ are units in $R$, and $c_{1j}, c_{2j}\in R$.
In particular, for $i=1,2$ and $3 \leq j \leq 2\ell+3$, there is $(s,t)\in R^2$ such that $L_{ij}(s,t)\in R^*$.
For $3 \leq i<j\leq 2r+1$, 
$$L_{ij}(s,t)=\alpha L_{i,j}(U(\ell),x).$$
does not depend on $s,t\in R$ and is a unit since $(U(\ell),x) \in U_{2\ell+2}(R^{2\ell})$.
Summarising, for every equation in (\ref{eqn:DetPfxZero}), there is $(s,t)\in R^2$ for which the left-hand side of that equation does not vanish in $F$.

From the computation of $L_{ij}(s,t)$ above, the matrix $M(U(\ell+1),\xi)$ is 
{\small
$$\left(
\renewcommand\arraystretch{2}
\begin{array}{cc|ccc}
0 &  \langle as+bt +c\rangle^{-1} & \cdots & \langle\delta_{1j}\rangle \langle a_{1j}s + c_{1j}\rangle^{-1}  \cdots  &\langle c_{1}\rangle^{-1}\\
\langle as+bt +c\rangle^{-1} & 0 & \cdots & \langle\delta_{2j}\rangle\langle b_{2j}t +c_{2j} \rangle^{-1}  \cdots &\langle -c_{2}\rangle^{-1}\\
\hline
\vdots&\vdots&&&\\
\langle\delta_{i1}\rangle \langle a_{1i}s + c_{1i}\rangle^{-1}&\langle\delta_{i2}\rangle\langle b_{2i}t +c_{2i} \rangle^{-1} &&\\
\vdots&\vdots&& \langle \alpha \rangle^{-1}\cdot M(U(\ell),x) &\\
\langle c_1\rangle^{-1}&\langle  c_2\rangle^{-1}&&&
\end{array}
\right)
$$}

\noindent
By assumption, it has determinant $g(s,t)=\det M(U(\ell+1),\xi)$ in $\sqrt{\ann(h[B])}$ for all $(\bar{s},\bar{t})\in F^2-S$ where $S$ is a finite union of affine subspaces of dimension $\leq 1$ defined by the equations (\ref{eqn:DetPfxZero}).
For $\gamma \in R^*$, consider the equation $\gamma = as+bt+c$ and note that for all but finitely many $\bar{\gamma}\in F^*$ the hyperplane $\bar{\gamma} = \bar{a}\bar{s}+\bar{b}\bar{t}+\bar{c}$ in $F^2$ is not entirely in $S$.
Then $s=a^{-1}(\gamma-c -bt)$ and $M(U(\ell+1),\xi)$ becomes
{\small
$$\left(
\renewcommand\arraystretch{2}
\begin{array}{cc|ccc}
0 & \langle \gamma\rangle^{-1} & \cdots &  \langle\delta_{1j}\rangle\langle \tilde{a}_{1j}t + \tilde{c}_{1j}\rangle^{-1}  \cdots  &\langle c_{1}\rangle^{-1}\\
\langle \gamma\rangle^{-1} & 0 & \cdots &  \langle\delta_{2j}\rangle\langle b_{2j}t +c_{2j} \rangle^{-1}  \cdots &\langle -c_{2}\rangle^{-1}\\
\hline
\vdots&\vdots&&&\\
\langle\delta_{i1}\rangle\langle \tilde{a}_{1i}t + \tilde{c}_{1i}\rangle^{-1} &\langle\delta_{i2}\rangle\langle b_{2i}t +c_{2i}\rangle^{-1}&&&\\
\vdots&\vdots&& \langle \alpha \rangle^{-1}\cdot M(U(\ell),x) &\\
\langle c_1\rangle^{-1}&\langle c_2 \rangle^{-1}&&&
\end{array}
\right)
$$
}

\noindent
where $\tilde{a}_{1j} = -a_{1j}b/a$ and $\tilde{c}_{1j} = c_{1j} + a_{1j}(\gamma-c)/a$.
Its determinant $f(t,\gamma) = g(a^{-1}(\gamma-c -bt),t)$ is in $\sqrt{\ann(z[B])}$ for all $\bar{t}\in F-S'$ for a finite set $S'\subset F$ (for fixed $\gamma$).
Since the coefficients $\tilde{a}_{1j}$ and $b_{2j}$ of $t$ are units in $R$, we can apply 
 the Limit Theorem \ref{thm:LimitThm} and find that $\lim_{t\to \infty}f(t,\gamma) \in \sqrt{\ann(h[B])}$ where
$$f(\gamma) = \lim_{t\to \infty}f(t,\gamma) = \det
\left(
\renewcommand\arraystretch{2}
\begin{array}{cc|ccc}
0 & \langle \gamma \rangle^{-1} &0 \cdots &0 \cdots  0&\langle c_{1}\rangle^{-1}\\
\langle \gamma \rangle^{-1} & 0 &0 \cdots & 0 \cdots 0&\langle -c_{2}\rangle^{-1}\\
\hline
0&0&&&\\
\vdots&\vdots&& \langle \alpha \rangle^{-1}\cdot M(U(\ell),x) &\\
0&0&&&\\
\langle c_{1}\rangle^{-1}&\langle c_{2}\rangle^{-1}&&&
\end{array}
\right)
$$
for all but finitely many $\bar{\gamma} \in F$.
Then
$$\langle \gamma\rangle^2 f(\gamma) = \det
\left(
\renewcommand\arraystretch{2}
\begin{array}{cc|ccc}
0 & 1 &0 \cdots &0 \cdots  0&\langle c_{1}\rangle^{-1}\langle \gamma \rangle\\
1 & 0 &0 \cdots & 0 \cdots 0&\langle -c_{2}\rangle^{-1}\langle \gamma \rangle\\
\hline
0&0&&&\\
\vdots&\vdots&& \langle \alpha \rangle^{-1}\cdot M(U(\ell),x) &\\
0&0&&&\\
\langle c_{1}\rangle^{-1}&\langle c_{1}\rangle^{-1}&&&
\end{array}
\right)
$$
is in $\sqrt{\ann(h[B])}$ for all but finitely an $\bar{\gamma} \in F$.
By the Limit Theorem \ref{thm:LimitThm}, the element
$$\lim_{\gamma \to 0}\langle \gamma\rangle^2 f(\gamma) = 
\det \left(
\renewcommand\arraystretch{2}
\begin{array}{cc|ccc}
0 & 1 &0 \cdots &0 \cdots  0&0\\
1 & 0 &0 \cdots & 0 \cdots 0& 0\\
\hline
0&0&&&\\
\vdots&\vdots&& \langle \alpha \rangle^{-1}\cdot M(U(\ell),x) &\\
0&0&&&\\
\langle c_{1}\rangle^{-1}&\langle c_{1}\rangle^{-1}&&&
\end{array}
\right)
$$
is also in $\sqrt{\ann(h[B])}$.
Hence, $-\langle \alpha\rangle^{-2\ell}\det M(U(\ell),x) \in \sqrt{\ann(z[B])}$ which implies
$\det M(U(\ell),x) \in \sqrt{\ann(h[B])}$ since $-\langle \alpha\rangle^{-2\ell}$ is a unit in $\Z_0[R]$.
This finishes the proof of (\ref{eqn:Surj:IndAss}) for $\ell = 1,...,r$.
In particular, it holds for $\ell =1$.

Finally, we investigate what (\ref{eqn:Surj:IndAss}) means for $\ell = 1$.
The given matrix
$$U(1)=\begin{pmatrix} 1 & 0 & b \\ 0 & a & c \end{pmatrix}$$
has $a,b,c\in R^*$ since it is in $U_3(R^2)$.
For $x=(s,t)\in R^2$, the matrix
$$(U(1),x) = \begin{pmatrix} 1 & 0 & b & s\\ 0 & a & c & t \end{pmatrix}$$
is in $U_4(R^2)$ if and only if $s,t,bt-cs\in R^*$.
Then $M(U(1),x) = (\langle\delta_{ij}{\det}^{-1}(U(1),x)^{\wedge}_{ij}\rangle)$ has determinant
$$f(s,t) = \det M(U(1),x) 
= \det \begin{pmatrix}
0 & \langle bt-cs\rangle^{-1} & \langle -as\rangle^{-1} & \langle -ab\rangle^{-1}\\
\langle bt-cs\rangle^{-1} & 0 & \langle -t\rangle^{-1} & \langle -c\rangle ^{-1}\\
 \langle as\rangle^{-1} & \langle -t\rangle^{-1} & 0 & \langle a \rangle^{-1}\\
  \langle -ab\rangle^{-1} &  \langle c\rangle ^{-1} & \langle a \rangle^{-1} & 0 
\end{pmatrix}
$$
in $\sqrt{\ann(h [B])}$ for all $s,t,\in R^*$ such that $bt-cs\in R^*$.
Setting $s=1$ then every $t\in R$ such that $\bar{t}\neq 0, \bar{c}/\bar{b}\in F$
has 
$$
f(1,t) 
= \det \begin{pmatrix}
0 & \langle bt-c\rangle^{-1} & \langle -a\rangle^{-1} & \langle -ab\rangle^{-1}\\
\langle bt-c\rangle^{-1} & 0 & \langle -t\rangle^{-1} & \langle -c\rangle ^{-1}\\
 \langle a\rangle^{-1} & \langle -t\rangle^{-1} & 0 & \langle a \rangle^{-1}\\
  \langle -ab\rangle^{-1} &  \langle c\rangle ^{-1} & \langle a \rangle^{-1} & 0 
\end{pmatrix}
$$
in $\sqrt{\ann(h [B])}$.
Since the coefficients $b$ an and $-1$ of $t$ are units in $R$, we can apply the Limit Theorem \ref{thm:LimitThm} and find that the element
$$\lim_{t \to \infty}f(1,t) = 
\det \begin{pmatrix}
0 &0 & \langle -a\rangle^{-1} & \langle -ab\rangle^{-1}\\
0 & 0 & 0 & \langle -c\rangle ^{-1}\\
 \langle a\rangle^{-1} & 0 & 0 & \langle a \rangle^{-1}\\
  \langle -ab\rangle^{-1} &  \langle c\rangle ^{-1} & \langle a \rangle^{-1} & 0 
\end{pmatrix}=
\langle ac\rangle^{-2}$$
is in $\sqrt{\ann(h [B])}$.
Since $\langle ac\rangle^{-2}$ is a unit in $\Z_0[R]$, the ideal $\sqrt{\ann(h [B])}$ is the unit ideal.
\end{proof}

\section{Homology stability}

In this section we prove the results announced in the Introduction.
The following proves Theorem \ref{thm:OptStabilityIntro}.

\begin{theorem}
\label{thm:OptStabilityText}
Let $R$ be a commutative local ring with infinite residue field and $n\geq 0$ an integer.
Then in the following sequence of integral homology groups, all maps are isomorphisms
%\begin{equation}
%\label{eqn:H2nStability}
$$H_{2n}(\Sp_{2n}R) \stackrel{\cong}{\longrightarrow}  H_{2n}(\Sp_{2n+1}R) \stackrel{\cong}{\longrightarrow} H_{2n}(\Sp_{2n+2}R) \stackrel{\cong}{\longrightarrow} \cdots
$$
%\end{equation}
and in the following sequence of integral homology groups, the first map is a surjection and all other maps are isomorphisms
%\begin{equation}
%\label{eqn:H2n+1Stability}
$$H_{2n+1}(\Sp_{2n+1}R) \twoheadrightarrow  H_{2n+1}(\Sp_{2n+2}R) \stackrel{\cong}{\longrightarrow} H_{2n+1}(\Sp_{2n+3}R) \stackrel{\cong}{\longrightarrow} \cdots.
$$
%\end{equation}
Moreover, inclusion of groups induces a surjection 
$$H_{2n+1}(\Sp_{2n}(R)) \twoheadrightarrow H_{2n+1}(\Sp_{2n+2}(R)).$$
In particular, $H_i(\Sp_{2n}(R),\Sp_{2n-2}(R))=0$ for all $i<2n$.
\end{theorem}

\begin{proof}
The case $n=0$ is clear, so assume $n\geq 1$.
The spectral sequence (\ref{eqn:E1spseq}) degenerates at the $E^2$-page (Proposition \ref{prop:Esurjective}).
By Lemma \ref{lem:AbuttVanishing}, we have $E^2_{p,q}(R^{2n})=0$ for $p+q<2n$.
Moreover, $0=d:\Z[\Skew^+_2(R)] \to \Z[\Skew_1^+(R)]$ forcing $d^1_{p,2}=0$ for all $p\in \Z$ (Lemma \ref{lem:commdiag}).
Therefore,
$d^1_{p,1}:E^1_{p,1}(R^{2n}) \to E^1_{p,0}(R^{2n})$ is an isomorphism for $p\leq 2n-2$ and a surjection for $p=2n-1$.
Hence,
$$H_{p}(\Sp_{2n-2}) \oplus \widetilde{H}_{p}(\Sp_{2n-1}) = H_{p}(\Sp_{2n-1}) \longrightarrow H_{p}(\Sp_{2n})$$
is an isomorphism for $p\leq 2n-2$ and a surjection for $p=2n-1$.
By Lemma \ref{lem:tildtoIszero}, the map is zero on the second summand.
In particular, 
$$\widetilde{H}_{p}(\Sp_{2n-1}) = 0 \hspace{4ex}\text{for}\hspace{2ex}p\leq 2n-2$$
and 
$$H_{p}(\Sp_{2n-2}) \stackrel{\cong}{\longrightarrow} H_{p}(\Sp_{2n})\hspace{4ex}\text{for}\hspace{2ex}p\leq 2n-2.$$
This proves the first string of isomorphisms and the second string of a surjection followed by isomorphisms in the theorem.
Using Lemma \ref{lem:tildtoIszero}, the surjectivity of 
$H_{2n-1}(\Sp_{2n-1}) \longrightarrow H_{2n-1}(\Sp_{2n})$
implies surjectivity of $H_{2n-1}(\Sp_{2n-2}) \longrightarrow H_{2n-1}(\Sp_{2n})$.
\end{proof}

Let $K^{MW}_*(R)$ be the Milnor-Witt $K$-theory ring of $R$ \cite[Definition 3.1]{Morel:book}, \cite[Definition 4.10]{myEuler}.
The following proves Theorem \ref{thm:OptStabilityIntro2} from the Introduction.

\begin{theorem}
\label{thm:ObstructionSurjection}
Let $R$ be a local ring with infinite residue field and $n\geq 1$ an integer.
Then the inclusions of groups $\Sp_{2r} \subset SL_{2r} \subset SL_{2r+1}$ induce a surjection
$$H_{2n}(\Sp_{2n}(R),\Sp_{2n-2}(R)) \twoheadrightarrow H_{2n}(SL_{2n}(R),SL_{2n-1}(R))=K^{MW}_{2n}(R).$$
\end{theorem}

\begin{proof}
Consider the string of maps 
{\small
$$H_2(\Sp_2(R))^{\otimes n} \to H_{2n}(\Sp_{2n}(R)) \to H_{2n}(\Sp_{2n}(R),\Sp_{2n-2}(R)) \to H_{2n}(SL_{2n}(R),SL_{2n-1}(R))$$
}

\noindent
in which the first map is induced by the block sum of matrices.
By \cite[Theorem 5.37 and proof]{myEuler}, the composition is the surjective multiplication map
$$K_2^{MW}(R)^{\otimes n} \twoheadrightarrow K_{2n}^{MW}(R).$$
It follows that the last map in the composition is surjective.
\end{proof}

\begin{remark}
\label{rmk:NotInjective}
Let $k$ be an infinite perfect field of characteristic not $2$ which is finitely generated over its prime field. 
Then neither of the two surjective maps
\begin{equation}
\label{eqn:NotInjective}
H_3(\Sp_2(k)) \twoheadrightarrow H_3(\Sp_4(k)),\hspace{2ex}\text{and }\hspace{2ex} H_4(\Sp_4(k),\Sp_2(k))\twoheadrightarrow K^{MW}_4(k)
\end{equation}
is injective.
For the first map, this follows from \cite[Theorem 7.4]{HutchinsonWendt} since that map factors through $H_3(B\Sp_2(k[\Delta^{\bullet}]))$ in view of the isomorphisms
$$H_3(B\Sp_4(k)) \cong H_3(B\Sp(k)) \cong H_3(B\Sp(k[\Delta^{\bullet}]))$$
resulting from Theorem \ref{thm:OptStabilityText} and homotopy invariance of symplectic $K$-theory for regular rings containing $1/2$.

If the second map in (\ref{eqn:NotInjective}) was an isomorphism, then the map
$$H_4(\Sp_4(k)) \to H_4(\Sp_4(k),\Sp_2(k))$$ 
would be surjective (see proof of Theorem \ref{thm:ObstructionSurjection}), and the long homology exact sequence for the pair $(\Sp_4(k),\Sp_2(k))$ would force the first map in (\ref{eqn:NotInjective}) to be injective.
\end{remark}

\bibliographystyle{plain}

%\bibliography{HSp}

\end{document}